\documentclass[12pt]{amsart}
\usepackage{geometry}   
\usepackage[colorlinks,citecolor = red, linkcolor=blue,hyperindex]{hyperref}
\usepackage{euscript,eufrak,verbatim, mathrsfs}
\usepackage[psamsfonts]{amssymb}
\usepackage[usenames]{color}
\usepackage{bbm}
\usepackage{graphicx}
 \usepackage{color}
 \usepackage{float}

 \usepackage{euscript}
\usepackage{helvet}         
\usepackage{courier}        
\usepackage{type1cm}        
\usepackage{multicol}        
\usepackage[bottom]{footmisc}

\newtheorem{theorem}{Theorem}[section]
\newtheorem*{theorem*}{Theorem B} 
\newtheorem{lemma}[theorem]{Lemma}

\newtheorem{proposition}[theorem]{Proposition}
\newtheorem{corollary}[theorem]{Corollary}
\newtheorem{definition}[theorem]{Definition}
\newtheorem*{definition*}{Definition}
\newtheorem*{remark*}{Remark}

\newtheorem*{observation*}{Observation}
\newtheorem{assumption}{Assumption}
\newtheorem*{assumption*}{Assumption}
\newtheorem*{question*}{Question}
\newtheorem{remark}[theorem]{Remark}
\newtheorem{example}[theorem]{Example}

\newcommand{\R}{\mathbb{R}}
\newcommand{\N}{\mathbb{N}}
\newcommand{\Z}{\mathbb{Z}}

\newcommand{\D}{\mathbb{D}}
\newcommand{\C}{\mathbb{C}}
\newcommand{\E}{\mathbb{E}}
\newcommand{\PP}{\mathbb{P}}

\newcommand{\X}{\mathfrak{X}}

\newcommand{\Conf}{\mathrm{Conf}}

\newcommand{\tr}{\mathrm{tr}}
\newcommand{\rank}{\mathrm{rank}}

\newcommand{\Unif}{\mathrm{Unif}}

\newcommand{\Det}{\mathrm{det}}
\newcommand{\Ran}{\mathrm{Ran}}
\newcommand{\Ker}{\mathrm{Ker}}
\newcommand{\Diff}{\mathrm{Diff}}

\newcommand{\an}{\text{\, and \,}}

\DeclareMathOperator*{\essinf}{\mathrm{essinf}}

\begin{document}

\title[Palm equivalence for Bergman DPP]{Equivalence of Palm measures for determinantal point processes governed by Bergman kernels}

\author
{Alexander I. Bufetov}
\address
{Alexander I. BUFETOV: 
Aix-Marseille Universit\'e, Centrale Marseille, CNRS, Institut de Math\'ematiques de Marseille, UMR7373, 39 Rue F. Joliot Curie 13453, Marseille, France;
Steklov Mathematical Institute of RAS, Moscow, Russia;
Institute for Information Transmission Problems, Moscow, Russia;
National Research University Higher School of Economics, Moscow, Russia;
The Chebyshev Laboratory, Saint-Petersburg State University, St-Petersburg, Russia}
\email{bufetov@mi.ras.ru, alexander.bufetov@univ-amu.fr}

\author
{Shilei Fan}
\address
{Shilei FAN: 
School of Mathematics and Statistics, Hubei Key Laboratory of Mathematical Sciences, Central  China Normal University,  Wuhan, 430079, China;
 Aix-Marseille Universit\'e, Centrale Marseille, CNRS, Institut de Math\'ematiques de Marseille, UMR7373, 39 Rue F. Joliot Curie 13453, Marseille, France}
\email{ slfan@mail.ccnu.edu.cn}

\author
{Yanqi Qiu}
\address
{Yanqi QIU: CNRS, Institut de Math{\'e}matiques de Toulouse, Universit{\'e} Paul Sabatier, 118 Route de Narbonne, F-31062 Toulouse Cedex 9, France}
\email{yqi.qiu@gmail.com}

\begin{abstract}
For a determinantal point process induced by the reproducing kernel of the weighted Bergman space $A^2(U, \omega)$ over a domain $U \subset \mathbb{C}^d$, we establish the mutual absolute continuity of  reduced Palm measures  of any order provided that the domain $U$ contains a non-constant bounded holomorphic function. The result holds in all dimensions.
The argument uses the $H^\infty(U)$-module structure of $A^2(U, \omega)$.  A corollary is the quasi-invariance of our determinantal point process  under the natural action of the group of compactly supported diffeomorphisms of $U$.  
\end{abstract}


\maketitle

\setcounter{equation}{0}

\section{Introduction}
\subsection{Formulation of the main results}
How does a point process change once conditioned to contain a particle at a given site? The question, one of the oldest  in the theory of point processes, goes back to the work of Palm \cite{palm} and Khintchine \cite{khin}. In this paper  we study Palm distributions of determinantal point processes governed by Bergman kernels. 

Let $U$ be a  non-empty connected open subset of the $d$-dimensional complex space $\C^d$.  
Let $\omega: U \rightarrow (0, \infty)$ be a Borel function.
Consider the weighted Bergman space 
\[
A^2(U, \omega): = \Big\{ \text{$f: U\rightarrow \C$ \Big| $f$ is holomorphic on $U$ and
 $\int_U | f(z)|^2 \omega(z) dV(z) < \infty$}\Big\}, 
\]
where  $V$ stands for the  Lebesgue measure on $\C^d$.  We assume  that $A^2(U, \omega) \ne \{0\}$  and  that  for any relatively compact subset $B\subset U$ satisfying $V(B)>0$ we have 
\begin{align}\label{weight-low}
\essinf_{z \in B} \omega(z)  > 0.
\end{align}
If \eqref{weight-low} holds, then the linear space $A^2(U, \omega)$ is closed in $L^2(U, \omega dV)$ and admits a reproducing kernel $K_{\omega}$,  called the {\it weighted Bergman kernel}. The operator, for which we keep the same symbol $K_\omega$, of orthogonal projection from $L^2(U, \omega dV)$ onto the subspace $A^2(U, \omega)$ is given by the formula
$$
(K_\omega \varphi)(z) =  \int_{U} K_\omega(z, u)  \varphi(u) \omega(u) dV(u) \,\, \text{for  $\varphi \in L^2(U, \omega dV)$.}
$$
The kernel $K_{\omega}$  induces a determinantal point process $\PP_{K_\omega}$ on $U$ (see  \S \ref{sec-example}  for  the definition of determinantal point processes and Palm measures). 
Recall that two measures are called {\it equivalent} if they are mutually absolutely continuous. The main result  of this paper is 
\begin{theorem}\label{thm-palm-eq}
If the domain $U$ admits a non-constant bounded holomorphic function, 
then the determinantal measure $\PP_{K_\omega}$ is equivalent to 
its arbitrary  reduced Palm measure, of any order.
\end{theorem}

Equivalence of Palm measures was previously  only obtained for the disk in one dimension: for the uniform weight in Holroyd and Soo  \cite{HolSoo} and  for a class of non-uniform weights in \cite{QB3}. Theorem \ref{thm-palm-eq}  holds for   arbitrary domains in all dimensions.
An immediate corollary of Theorem \ref{thm-palm-eq} is that the measure $\PP_{K_\omega}$ is quasi-invariant under the action of the group 
of compactly supported diffeomorphisms of $U$, see Corollary \ref{quasi-inv} below.  

Recall that the Hardy space $H^\infty(U)$ is defined by the formula
\[
H^\infty(U): = \Big\{ \text{$f: U\rightarrow \C$ \Big| $f$ is holomorphic on $U$ and $\sup_{z\in U} | f(z)| < \infty$} \Big\}.
\]
At the centre of the  proof of  Theorem \ref{thm-palm-eq} lies  the  $H^\infty(U)$-module structure of the space
$A^2(U, \omega)$.  To show how this module structure is used, we give a quick outline for the proof of {\it deletion tolerance} for our point process $\PP_{K_\omega}$. Deletion tolerance 
claims the positivity of the conditional probability, with respect to fixed exterior, of the absence of particles in a bounded domain.
This gap probability is equal to $\det(1-K)$, where $K$ is the corresponding conditional kernel, whose existence follows from the results of  \cite{BQS16}. If the gap probability is zero, then $K$ has a nonzero  invariant vector. By compactness of $K$, the invariant subspace of $K$ has finite dimension. We show, however, that the subspace of invariant vectors for the conditional operator $K$ is preserved under multiplication by bounded holomorphic functions and thus has infinite dimension, a contradiction.

Theorem \ref{thm-palm-eq} has a natural analogue for spaces of $q$-holomorphic functions on  domains in the  complex plane.
Let $D \subset \C$ be an open connected subset, and let  $q\in \N$. A continuous function $f: D \rightarrow\C$ is called $q$-holomorphic if it satisfies the  partial differential equation
$$
\bar{\partial}_z^q f(z) = 0, \quad z \in D,
$$
 understood as an equality of  distributions of Sobolev and Schwartz; 
here $\bar{\partial}_z = \frac{1}{2}( \partial_x + i \partial_y)$ in the standard coordinate system $z = x + i y$.   Define 
\[
A^2_q(D, \omega): = \Big\{\text{$f: D \rightarrow \C$ \Big| $f$  is $q$-holomorphic on $D$ and $\int_D | f(z)|^2 \omega(z) d V(z) < \infty$}\Big\}. 
\]
Under the assumption \eqref{weight-low}, the space $A^2_q(D, \omega)$ is closed in $L^2(D, \omega dV)$ and admits a reproducing kernel.  These spaces and the corresponding determinantal point processes have been studied by Haimi and Hedenmalm \cite{Haimi-Hedenmalm-JSP, Haimi-Hedenmalm}.  Note that $A_q^2(D,\omega)$ is an $H^\infty(D)$-module.  Let  $H \subset A_q^2(D,\omega)$ be a non-zero closed subspace which is an $H^\infty(D)$-submodule of $A_q^2(D,\omega)$, and let $\Pi_H$ be the orthogonal projection from $L^2(D, \omega dV)$ onto $H$.  The opertor $\Pi_H$ is locally of  trace class and induces a 
determinantal point process $\PP_{\Pi_H}$ on $D$.

\begin{theorem}\label{thm-q-hol} 
If $H^{\infty}(D)$  contains  a non-constant  function, then  the determinantal measure $\PP_{\Pi_H}$ is equivalent to 
its arbitrary  reduced Palm measure, of any order. 
\end{theorem}
Note that if  $\C \setminus \overline{D}$ has nonempty interior, then $H^\infty(D)$  contains a non-constant  function.

\subsection{Examples of Bergman kernels}

In a few cases weighted Bergman kernels can be calculated explicitly (see Krantz \cite[Chapter 1]{Krantz} for more details). To stress dependence on $U$, we write $K_{U, \omega}$ for $K_{\omega}$, and, in the case $\omega \equiv 1$, for brevity we write  $K_U$ for $K_{U, 1}$.  
\begin{example}\label{ex-classical}
The weighted Bergman kernel for the classical weight on the unit disk $\D \subset \C$ given by $\omega_\alpha (z) =  ( 1 + \alpha) ( 1 - | z|^2)^\alpha$ with $- 1 < \alpha < \infty$, is given by 
\[
K_{\omega_\alpha} (z, w) = \frac{1}{\pi ( 1 - z \bar{w})^{2 + \alpha}}. 
\]
\end{example}

\begin{example}
The Bergman kernel for the circular annulus $A_\rho = \{z \in \C: \rho < | z| <1 \}$ was calculated by Bergman \cite[p. 10]{Bergman-book}:  
\[
K_{A_\rho}(z, w) = \frac{1}{\pi z \bar{w}} \Big(   \wp (\ln (z\bar{w})) + \frac{\eta_1}{\pi i} - \frac{1} {2\ln \rho} \Big), 
\]
where $\wp$ is the Weierstrass function with the periods $\pi i$ and $\ln \rho$, and $\eta_1$ is the half-increment of the Weierstrass $\zeta$-function related to the period $\pi i$.
\end{example}

\begin{example}
The Bergman kernel of the polydisk $\D^d$ is given by 
\[
K_{\D^d} (z, w)  = \prod_{j = 1}^d K_\D (z_j, w_j) = \prod_{j=1}^d \frac{1}{\pi ( 1 - z_j \bar{w}_j)^2}. 
\]
\end{example}

\begin{example}
The  Bergman kernel of the unit ball $\mathbb{B}_d =\{z \in \C^d: \sum_{j=1}^d | z_j|^2 < 1\}  $ is given by 
\[
K_{\mathbb{B}^d} (z, w) = \frac{1}{V(\mathbb{B}_d)} \frac{1}{( 1 - z \cdot \bar{w})^{d + 1}}, 
\]
where $V(\mathbb{B}_d)$ denotes the volume of the unit ball $\mathbb{B}_d$ and $z \cdot \bar{w}: = \sum_{j=1}^d z_j \cdot \bar{w}_j$.
\end{example}

 Theorem \ref{thm-palm-eq} applies to all these examples.

\subsection{Historical remarks}\label{sec-his}

Recall that the Bergman kernel $K_\D$ of the unit disk $\D\subset \C$ is given by the formula
\[
K_\D (z, w) = \frac{1}{\pi ( 1 - z \bar{w})^2}. 
\]
A remarkable theorem of Peres and Vir{\' a}g \cite{PV-acta} is that  the determinantal point process $\PP_{K_\D}$  describes the zero set $Z(f) = \{z \in \D: f(z)  = 0\}$ of the Gaussian analytic function 
\[
f( z) = \sum_{n=0}^\infty g_n z^n,
\]
 where $(g_n)_{n=0}^\infty$ is a sequence of independent standard complex Gaussian random variables. The equivalence of the determinantal point process $\PP_{K_\D}$ to its Palm measures  is due to Holroyd and Soo \cite{HolSoo}.   The argument of Holroyd and Soo uses the  zero set representation of our process $Z(f)$ and, as far as we are aware,   can not be generalized.
For the unit disk $\D$,  Theorem \ref{thm-palm-eq}  was established for all the determinantal point processes $\PP_{K_{\D, \omega}}$ induced by the weighted Bergman kernels $K_{\D, \omega}$ in \cite{QB3} under the additional assumption
\begin{align}\label{w-ass}
\int_\D (1 - | z|)^2 K_{\D, \omega}(z,z) \omega(z)dV(z)< \infty.
\end{align}
Note that all the classical weights $\omega_\alpha$ introduced in Example \ref{ex-classical} satisfy the condition \eqref{w-ass}. Moreover, the Radon-Nikodym derivative between the process $\PP_{K_{\D, \omega}}$ and its Palm measure is explicitly computed in \cite{QB3} and represented as a  regularized multiplicative functional corresponding to a divergent Blaschke product.

Even in one dimension, however, the formalism of \cite{QB3} does not seem to be applicable to bounded
multiply connected domains, for example, the annuli. In fact, while Theorem \ref{thm-palm-eq} applies to
 the reproducing kernel of the Hilbert  space of square-integrable holomorphic functions on any annulus, we are not able to compute  the Radon-Nikodym derivatives explicitly in this case. 

In the case of domains having the Liouville property, that is, admitting only constant bounded holomorphic functions, the behaviour of determinantal point processes governed by reproducing kernels of Hilbert spaces of holomorphic functions is quite different. 
For example, consider $U  = \C$ and  recall that the Ginibre point process induced by the reproducing kernel of  the space 
$
A^2(\C, e^{-|z|^2} dV). 
$
Ghosh and Peres \cite{Ghosh-rigid} proved that the Ginibre point process  is {\it number rigid}, that is, for any relatively compact $B\subset \C$, the number of particles inside $B$ is almost surely measurably determined by the configuration outside $B$. In particular, number rigidity implies that all reduced Palm measures of different orders are singular. Conversely,  Osada and Shirai \cite{OS}  proved that reduced Palm measures of the same order are mutually absolutely continuous for the Ginibre point process.  For generalized Ginibre point processes, corresponding to the weighted Fock spaces 
$
A^2(\C, e^{-2\psi(z)} dV)$ 
with $\psi$ a $C^2$-smooth function whose Laplacian $\Delta \psi$ satisfies $m \le \Delta \psi \le  M$ for some $m, M > 0$, number rigidity,   equivalence of Palm measures of the same order and  quasi-invariance under the natural action of  the group $\Diff_c(\C)$ of compactly supported diffeomorphisms of $\C$  are proved in \cite{QB3} (see \S \ref{sec-sym} for the precise definition of  the group $\Diff_c(\C)$). 
\begin{question*}
Let $U\subset \C^d$ be a connected open subset such that $H^\infty(U)$ contains only constant functions. Does it follow that  the determinantal point process $\PP_{K_{U, \omega}}$ is number rigid? 
\end{question*}

\subsection{Outline of the proof of Theorem \ref{thm-palm-eq}}

Given  a positive integer $\ell \in \N$ and an $\ell$-tuple  $\mathfrak{p} = (p_1, \cdots, p_\ell) \in U^\ell$ of { distinct} points  in $U$, we denote by $\PP_{K_\omega}^{\mathfrak{p}}$ the { reduced Palm measure}  of $\PP_{K_\omega}$ corresponding to the positions $p_1, \cdots, p_\ell$ (see \S\ref{sec-example} below for the precise definition).   The number $\ell$ is called the order of $\PP_{K_\omega}^{\mathfrak{p}}$.  
 Let  $\ll$ denote the relation of absolute continuity.  Our aim is to prove the 
 relations $
\PP_{K_\omega}^{\mathfrak{p}} \ll \PP_{K_\omega}, \quad \PP_{K_\omega} \ll \PP_{K_\omega}^{\mathfrak{p}}. 
$

We prove the relation $\PP_{K_\omega}^{\mathfrak{p}} \ll \PP_{K_\omega}$ by showing in Lemma \ref{lem-del} that the measure $\PP_{K_\omega}$ is 
{\it deletion tolerant}: given a configuration and a bounded set, 
the event that the bounded set contains no particles  has positive conditional probability with respect to fixing the configuration in the outside of our bounded set (see Definition \ref{def-del}). Equivalently, see Holroyd and Soo \cite[Theorem 1.1]{HolSoo}, the probability law of a random configuration $\mathscr{X}$ is deletion tolerant if and only if  measurable removal of  a random finite subset from $\mathscr{X}$ yields a random configuration whose law is absolutely continuous with respect to the original law. A general  Proposition \ref{prop-del-ab}  shows that  deletion tolerance of $\PP_{K_\omega}$ implies  $\PP_{K_\omega}^{\mathfrak{p}} \ll \PP_{K_\omega}$.  

The deletion tolerance property for $\PP_{K_\omega}$ is established  using the module structure of our 
underlying Hilbert  space $A^2(U, \omega)$ over the algebra $H^\infty(U)$.  For any Borel subset $W\subset U$ and $\PP_{K_\omega}$-almost every configuration $\X$,  denote by $\PP_{K_\omega}(\cdot| \X, W)$ the conditional measure of $\PP_{K_\omega}$  with respect to the condition that the configuration on $W$ coincides with $\X|_W  : = \X \cap W$.  Denote by $\#_W$ the map which associates a configuration to its number of particles inside $W$. We need to show that for any relatively compact $B\subset U$ with positive Lebesgue measure, $\PP_{K_\omega}(\#_{B}=0|\X,B^{c})>0$ for $\PP_{K_\omega}$-almost every $\X$.
An explicit description of the conditional measure $\PP_{K_\omega}(\cdot| \X, B^c)$ was recently obtained in \cite{BQS16}: for $\PP_{K_\omega}$-almost every $\X$, the conditional measure $\PP_{K_\omega}(\cdot|\X, B^c)$ is  determinantal and is induced by an explicitly described trace class positive contraction $K_\omega^{[\X, B^c]}$, that is, 
\begin{align*}
\PP_{K_\omega}(\cdot| \X, B^c)= \PP_{K_\omega^{[\X, B^c]}}. 
\end{align*}
Since
\[
\PP_{K_\omega}(\#_{B}=0|\X,B^{c})= \det  (1 - \chi_B K_\omega^{[\X, B^c]} \chi_B), 
\]
we show in Lemma \ref{lem-strict-c} that our compact operator $K_\omega^{[\X, B^c]}$ is { strictly contractive} for $\PP_{K_\omega}$-almost every configuration $\X \in \Conf(U)$, in other words,  that $1$ is not an eigenvalue of $K_\omega^{[\X,B^{c}]}$. The key   Lemma \ref{lem-module} shows that the  space 
\begin{equation}\label{inv-space-kom}
 \{h\in L^{2}(B): K_\omega^{[\X, B^{c}]}(h)=h\}
 \end{equation}
 is an $H^\infty(U)$-module. If $H^\infty(U)$ contains a non-constant function, then the space (\ref{inv-space-kom}) is either zero or  has infinite dimension, the latter impossible by compactness of  $K_\omega^{[\X, B^c]}$.

The argument is concluded using a {\it monotone coupling} of $\PP_{K_\omega}^{\mathfrak{p}}$ and $\PP_{K_\omega}$, that is, a measure on the Cartesian square on the space of configurations supported on the set of pairs of configurations one of which contains the other, and with projection marginals $\PP_{K_\omega}^{\mathfrak{p}}$ and $\PP_{K_\omega}$ respectively (see \S  \ref{sec-mc}  for precise statements) .

The proof of the relation $\PP_{K_\omega} \ll \PP_{K_\omega}^{\mathfrak{p}}$ requires greater effort.  We start by  establishing,  
 in Lemma \ref{lem-num-ins},  the {\it number insertion tolerance} for our reduced Palm measure $\PP_{K_\omega}^{\mathfrak{p}}$. Number insertion tolerance, see Definition \ref{def-num-ins}, is dual to 
deletion tolerance and means that, given a configuration and a bounded set, 
the event that the bounded set contains at least one particle has positive conditional probability with respect to  fixing the configuration in the outside of our bounded set. The proof of number insertion tolerance uses both the module 
structure of our underlying Hilbert space and  the  {\it real analyticity} of our reproducing kernels and the well chosen conditional kernels $(K_\omega^{\mathfrak{p}})^{[\X, B^c]} (z, w)$. In particular, in  Proposition \ref{prop-anal}, we prove a general  result which states that any determinantal point process on a connected domain,  induced by a real analytic kernel which represents an infinite rank trace class positive contractive operator,  is indeed number insertion tolerant.

Number insertion tolerance is however too weak for our purposes, and we need the stronger notion of  {\it insertion tolerance}, see Definition \ref{def-ins}. The key additional property we need to prove is that our point process is  {\it diffusive} in that the conditional measure  in a bounded domain $B$, with respect to 
fixed exterior, is absolutely continuous with respect to the Lebesgue measure, or, more formally, the  compound Lebesgue measure on the disjoint union of Cartesian powers of our bounded domain $B$. The probability law of a random configuration $\mathscr{X}$ is insertion tolerant if and only if any measurable and {\it diffusive} addition of a  finite subset to  $\mathscr{X}$ yields a new random configuration whose law is absolutely continuous with respect to the original one.   
The key consideration here is again the real analyticity of our conditional kernels established in Lemma \ref{lem-real-anal}.  We emphasize that the proof of real analyticity of conditional kernels in Lemma \ref{lem-real-anal} uses the complex analyticity of the functions in our Bergman spaces, see Remark \ref{real-complex-rem}. 
In  Lemma  \ref{lem-nice} and Lemma \ref{lem-abs}, we use  real analyticity of the kernel   to derive the insertion tolerance from  the number insertion tolerance. 

The final step of our argument  is  that for any monotone coupling of $\PP_{K_\omega}^{\mathfrak{p}}$ and $\PP_{K_\omega}$, the difference random configuration is {\it conditionally diffusive}, see Lemma \ref{lem-re-f} for the precise statement.  The  insertion tolerance of the   reduced Palm measure $\PP_{K_\omega}^{\mathfrak{p}}$ then yields the desired relation $\PP_{K_\omega} \ll \PP_{K_\omega}^{\mathfrak{p}}$.

\begin{remark}
Our argument in fact yields a stronger claim.
Let $L \subset A^2(U,\omega)$ be a non-zero closed subspace which is an $H^\infty(U)$-sub-module of $ A^2(U,\omega)$. Let $\Pi_L$ be the operator of orthogonal projection from $L^2(U, \omega dV)$ onto $L$, then  $\Pi_L$ is locally trace class and induces a determinantal measure $\PP_{\Pi_L}$ on $U$. If $H^{\infty}(U)$  contains  a non-constant  function, then the   measure $\PP_{\Pi_L}$ is equivalent to 
its arbitrary  reduced Palm measure, of any order.
\end{remark}

\section{Preliminaries}\label{sec-example}

\subsection{Spaces of configurations and point processes}
Let $E$ be a locally compact Polish space and let $\mu$ be  a $\sigma$-finite  Radon measure on $E$ such that the support of $\mu$ is the whole space $E$.

A (locally finite) configuration on $E$ is a collection of points of $E$, possibly with multiplicities and considered without regard to order,  such that any relatively compact subset of $E$ contains only finitely many points.  Points in a configuration will also be called particles. A configuration is called simple if all its particles have multiplicity one.  In this paper, we use Gothic  letters such as $\X, \mathfrak{Y}, \mathfrak{Z}$  to denote configurations.  Let $\Conf(E)$ denote the space of all configurations on $E$: 
\[
\Conf(E): = \Big\{\X \subset E \Big|  \text{$\#( \X \cap B)< \infty$ for any relatively compact $B\subset E$} \Big\}.
\] 
Equivalently, a configuration $\X \in \Conf(E)$  is a purely atomic Radon measure
$$
m_\X : = \sum_{x \in \X} \delta_x,
$$
where $\delta_x$ is the Dirac mass on the point $x$. The space $\Conf(E)$ is thus a subset of the space $\mathfrak{M}(E)$ of Radon measures on $E$ and becomes a complete separable metric space with respect to the vague topology on $\mathfrak{M}(E)$. We equip  $\Conf(E)$ with its Borel sigma-algebra. Note that  the Borel sigma-algebra on $\Conf(E)$ is the smallest sigma-algebra on $\Conf(E)$ making all the mappings $\X \mapsto \#(\X \cap B)$ measurable, with $B$ ranging over relatively compact Borel subsets of $E$. For further background on the general theory of point processes, see  Daley and Vere-Jones \cite{DV-1}, Kallenberg \cite{Kallenberg}.

 A Borel probability measure on $\Conf(E)$ is called a {\it point process} on $E$. A point process on $E$ is called simple if it  is supported  on the set of simple configurations on $E$. 
 A measurable map 
 $
 \mathscr{X}: \Omega \rightarrow \Conf(E)
 $
 defined on a probability space $(\Omega, \mathcal{B}, \bold{P})$ is called  a {\it random configuration} on $E$.
 Two random configurations $\mathscr{X}, \mathscr{Y}$ on $E$ defined on a common probability space are called {\it coupled random configurations}. 
In particular, it is convenient for us to distinguish the random configuration and the law of our random configuration, the  point process. 

 \subsection{Reduced Palm measures of point processes}

For a simple point process $\PP$ on $E$, recall that it is said to admit {\it $k$-th correlation measure} $\rho_k$ on $E^k$ if for any continuous compactly supported function $\varphi: E^k \rightarrow \C$ we have 
\begin{align*}
\int\limits_{\Conf(E)} \sum_{x_1, \dots, x_k \in \X}^*  \varphi(x_1, \dots, x_k) \PP(d \X) =  \int\limits_{E^k} \varphi(q_1, \dots, q_k) d \rho_k (q_1, \dots, q_k),
\end{align*}
where $\sum\limits^{*}$ denotes the sum over all ordered $k$-tuples of {\it distinct} points $(x_1, \dots, x_k) \in \X^k$.

For  a simple point process $\PP$ on $E$ admitting $k$-th correlation measure $\rho_k$ on $E^k$,  one can define, for $\rho_k$-almost every $\mathfrak{p} = (p_1, \dots, p_k) \in E^k$ of distinct points in $E$,  a point process  on $E$,  denoted by $\PP^{\mathfrak{p}}$ and called the {\it $k$-th order reduced Palm measure} of $\PP$ corresponding to the positions $p_1, \cdots, p_k$, by the following disintegration formula: for any non-negative Borel function $u: \Conf(E) \times E^k\rightarrow \R $, 
\begin{align*}
\int\limits_{\Conf(E)}  \sum_{p_1, \dots, p_k \in \X}^{*} u(\X; \mathfrak{p}) \PP(d \X)  =    \int\limits_{E^k} \rho_k(d\mathfrak{p}) \!\int\limits_{\Conf(E)} \!  u (\X \cup \{p_1, \dots, p_k\};  \mathfrak{p})  \PP^{\mathfrak{p}}(d\X).
\end{align*}
Informally,  $\PP^{\mathfrak{p}}$ is the conditional distribution of $\mathscr{X} \setminus \{p_1, \dots, p_k\}$ on $\Conf(E)$ conditioned to the event that  the configuration $\mathscr{X}$ has a particle at the positions $p_1, \dots, p_k $, provided that $\mathscr{X}$ has distribution $\PP$.  For more details on reduced Palm measures of point processes, see Section 12.3 in Kallenberg \cite{Kallenberg} 
or Section 2 in \cite{BQS}.

\subsection{Conditional measures of point processes}

  Let $\PP$ be a point process on $E$.  Let $W \subset E$ be a Borel subset. Consider the restriction mapping $\pi_W: \X \to \X|_W$ from
$\Conf(E)$ to $\Conf(W)$. For $\PP$-almost every configuration $\X$, the conditional measure $\PP(\cdot| \X|_W)$ of $\PP$ with respect to the condition that the restriction of the configuration onto $W$ coincides with $\X|_W$  is the conditional measure, in the sense of Rohlin \cite{Roh-meas},  supported on the  fibre
\begin{align}\label{fiber-form}
\{\mathfrak{Y}\in \Conf(E):  \mathfrak{Y}|_W = \X|_W\}. 
\end{align}
These conditional measures are characterized by the disintegration formula
 \begin{align}\label{f-dis-1}
\PP = \int\limits_{\Conf(W)}  \PP(\cdot| \X|_W)   (\pi_{W})_{*}\PP( d\X|_W) =  \int\limits_{\Conf(E)}  \PP(\cdot| \X|_W)   \PP( d\X).
\end{align}

It is however more convenient for us to consider our conditional measures as measures on the space $\Conf(E\setminus W)$, and we  set
\[
\PP(\cdot| \X, W) : = (\pi_{E\setminus W})_{*}\big[ \PP(\cdot | \X|_W)\big].
\]
With a slight abuse of terminology, the measures $\PP(\cdot| \X, W)$ will still be called  conditional measures of $\PP$. 
More precisely,  set $\Conf_E(W)=\pi_W(\Conf(E)), \Conf_E(E\setminus W) = \pi_{E\setminus W}(\Conf(E))$ and  write
\begin{equation}\label{prod-ident}
\Conf(E)  \xrightarrow{\quad \simeq\quad } \Conf_E(E\setminus W) \times \Conf_E(W),
\end{equation}
identifying a configuration $\X \in \Conf(E)$ and the pair $(\X|_{E\setminus W}, \X|_{W})$.  The fibre  \eqref{fiber-form} of the map $\pi_W$ is then  identified with the set
\[
\Conf_E(E\setminus W) \times \{\X|_W\}, 
\]
where $\{\X|_W\}$ is a singleton. The conditional measure $\PP(\cdot| \X|_W)$  is identified under (\ref{prod-ident}) with the probability measure 
\[
\PP(\cdot | \X, W) \otimes \delta_{\X|_W}. 
\] 
The disintegration formula (\ref{f-dis-1}) can now be written in the form
\begin{align}\label{f-dis-int}
\PP =  \int\limits_{\Conf(W)}  \Big(\PP(\cdot| \mathfrak{Z}, W) \otimes \delta_{\mathfrak{Z}}\Big) [\PP]_W(d \mathfrak{Z}) = \int\limits_{\Conf(E)}  \Big(\PP(\cdot| \X, W) \otimes \delta_{\X|_W}\Big) \PP(d \X).
\end{align}
where $ [\PP]_W: = (\pi_W)_{*}(\PP)$ is a probability measure supported on $\Conf_E(W) \subset \Conf(W)$ and thus can be considered as a probability measure on $\Conf(W)$.

\subsection{Determinantal point processes}\label{sec-DPP}

  Let $K$ be a {\it locally trace class positive contractive} operator on the complex Hilbert space $L^2(E,\mu)$. The local trace class assumption implies that $K$ is an integral operator and by slightly abusing the notation, we denote the kernel of the operator $K$ again by $K(x, y)$.  By  a theorem obtained by Macchi  \cite{Macchi-DP} and  Soshnikov \cite{Soshnikov-DP}, as well as by Shirai and Takahashi  \cite{ST-DPP}, the kernel $K$ induces a unique simple point process $\PP_K$ on $E$ such that for any positive integer $l\in \N$, the $l$-th correlation measure of $\PP_K$ exists and is given by 
\[
\rho_l  = \det(K(x_i, x_j))_{1 \le i, j \le l}  \cdot \mu(dx_1) \cdots \mu(dx_l) .
\]
The point process $\PP_{K}$ is called  the determinantal point process (or determinantal measure) induced by the kernel $K$. 

Recall that the determinantal point process $\PP_K$ is also characterized by the formula
\[
\E_{\PP_K}\Big(\prod_{x \in \mathfrak{X}} g (x) \Big) = \det( 1 + (g -1) K\chi_{B}),
\]
where $g$ is any bounded Borel function $g: E\rightarrow \C$ such that $g-1$ is supported on a relatively compact subset $B\subset E$. Here $\det$ stands for the Fredholm determinant. In particular, by taking $g =1 -  \chi_{B}$, we obtain the identity for the gap probability 
\begin{align}\label{id-gap}
\PP_K(\{\X \in \Conf(E): \X\cap B = \emptyset\}) = \det( 1 -\chi_B K\chi_{B}). 
\end{align}

\subsection{Palm measures of determinantal point processes}
For any $\ell \in \N$ and any $\ell$-tuple  $\mathfrak{p} = (p_1, \cdots, p_\ell) \in E^\ell$ of {\it distinct} points  in $E$, by the Shirai-Takahashi Theorem \cite[Theorem 6.5, Corollary 6.6]{ST-palm}, the reduced Palm measure $\PP_{K}^\mathfrak{p}$ is again a determinantal point process whose kernel is given explicitly:  setting 
\[
p_0= z, q_0 = w, q_i = p_i \ \mathrm{ for} \  1\le i \le \ell
\]
and denoting
 \begin{align}\label{def-Palm-kernel}
 K^{\mathfrak{p}}(z, w) : = \displaystyle \frac{\Det[K(p_i, q_j)]_{0 \le i, j \le \ell}
}{
 \Det[K(p_i, p_j)]_{1\le i, j \le \ell}
}, 
\end{align}
by  Shirai-Takahashi \cite[Corollary 6.6]{ST-palm}, we have 
\[
\PP_{K}^\mathfrak{p} = \PP_{K^\mathfrak{p}}.
\]

Recalling that our weighted Bergman kernel $K_\omega$ corresponds to the orthogonal projection from $L^2(U,\omega dV)$ onto $A^2(U, \omega)$, we see that the kernel $K_\omega^\mathfrak{p}$ defined by the formula \eqref{def-Palm-kernel}  corresponds to the orthogonal projection from $L^2(U, \omega dV)$ onto 
\begin{align}\label{berg-palm-sp}
A^2(U, \omega; \mathfrak{p}) :  = \{\varphi\in A^2(U, \omega)| \varphi(p_1) = \cdots = \varphi(p_\ell) = 0\}.
\end{align}

\subsection{Quasi-symmetries}\label{sec-sym}

 Recall that a $C^1$-diffeomorphism $F: U\rightarrow U$ is called {\it compactly supported}, if  the set $\{ z \in U: F (z) \ne z\}$ is relatively compact in $U$. Let $\Diff_c(U)$ denote the group of compactly supported diffeomorphisms of $U$. Consider the natural action of  $\Diff_c(U)$  on $\Conf(U)$ defined by the map $\Diff_c(U) \times \Conf(U) \rightarrow \Conf(U)$:  
\[
(F, \X) \mapsto F(\mathfrak{X}) : = \{ F(z): z \in \X \}.
\]

\begin{corollary}\label{quasi-inv}
The determinantal point process $\PP_{K_\omega}$ is $\Diff_c(U)$- quasi-invariant.  More precisely, $F_*(\PP_{K_\omega})$ is equivalent to  $\PP_{K_\omega}$ for all $F\in \Diff_c(U)$,   where $F_*(\PP_{K_\omega})$ is the image measure of $\PP_{K_\omega}$ under the map $F: \Conf(U)\rightarrow \Conf(U)$. 
\end{corollary}
\begin{proof}
Theorem \ref{thm-palm-eq}  implies in particular that all the reduced Palm measures of  $\PP_{K_\omega}$ { of the same order} are equivalent, and the desired  $\Diff_c(U)$-quasi-invariance of $\PP_{K_\omega}$ follows by  \cite[Proposition 2.19]{BQS}. 
\end{proof}

\subsection{Conditional measures of determinantal point processes}\label{sec-cm}

\subsubsection{Conditional kernels}\label{sec-cond-rep}
We shall need  the explicit description of the conditional measure $\PP_K(\cdot| \X, B^c)$  obtained in \cite{BQS16}. While the results in \cite{BQS16} are obtained for arbitrary locally trace class positive contractions, for the purposes of this paper we only need the case when the locally trace class positive contraction $K$ is an orthogonal projection, or, in other words, $K$ is the reproducing kernel of a reproducing kernel Hilbert space  $H \subset L^2(E, \mu)$ which is assumed to be locally trace class.

For a simple configuration $\X\in \Conf(E)$ and a Borel subset $W\subset E$,  set
\[
H(\X):=\{ \varphi \in H: \varphi|_\X = 0 \}; \ 
 \chi_{W}H(\X):=\{ \chi_{W}\cdot \varphi: \varphi \in H(\X)\}.
\]

Given  a  relatively compact Borel  subset $B\subset E$,   let $E_n\subset B^c$ be Borel subsets,  relatively compact in $E$, such that 
\[
E_1\subset E_2 \subset \cdots \subset E_n  \subset \cdots \hbox{ and }  \bigcup_{n}E_n= B^c.
\]
We also denote $F_n=E \setminus(B\cup E_n)$.

For $\PP_{K}$-almost every $\mathfrak{X}\in \Conf(E)$, denote by $K^{[\X, E_n]}$ the orthogonal  projection, from $L^2(E, \mu)$ onto  subspace 
  \[
\overline{\chi_{B\cup F_n}H(\mathfrak{X}\cap E_n)}^{L^2(E, \mu)},
  \]
  the closure of the linear subspace $\chi_{B\cup F_n}H(\mathfrak{X}\cap E_n)$. 
By \cite[Theorem 1.4]{BQS16}, for $\PP_{K}$-almost every configuration $\X\in \Conf(E)$, the conditional measure $\PP_K(\cdot| \X,B^{c})$ is again determinantal; 
 the operators $K^{[\X, E_n]}$ are locally trace class and there exists a  trace class positive contraction $K^{[\X, B^{c}]}$ on $L^{2}(E, \mu)$, such that
  \begin{align} \label{operator-con}
  \chi_{B}K^{[\X, E_n]}\chi_{B} \xrightarrow[\hbox{in trace class norm}]{n\to \infty}K^{[\X, B^{c}]}
  \end{align}
and the limiting operator $K^{[\X,B^{c}]}$  induces the determinantal measure $\PP_K(\cdot| \X,B^{c})$, that is, 
  \begin{align}\label{cond-m-f}
  \PP_K(\cdot| \X,B^{c}) = \PP_{K^{[\X,B^{c}]}}.
  \end{align}

\begin{remark}
The convergence \eqref{operator-con} implies in particular $K^{[\X, B^c]} = \chi_B K^{[\X, B^c]} \chi_B$. Therefore, in what follows, we also consider $K^{[\X, B^c]}$ as an operator acting on the subspace $L^2(B, \mu) \subset L^2(E, \mu)$. 
\end{remark}

\subsubsection{Trace class kernel case}\label{sec-trace}

If $K$ is a trace class positive contraction  on $L^2(E, \mu)$, then  the determinantal point process $\PP_K$ is supported  on the set of finite configurations: 
$
\PP_K(\#_E <\infty) =1. 
$
The formulas for conditional measures  are then particularly simple  \cite[Proposition 2.5]{BQS16}: for any Borel subset $W\subset E$ and $\PP_K$-almost every configuration $\X\in \Conf(E)$,  the conditional measure $\PP_K(\cdot | \X, W^c)$ is a determinantal point process and is induced by the kernel 
\begin{align*}
K^{[\X, W^c]}:  &= \chi_{W} K^{\X \cap W^c} (1 - \chi_{W^c} K^{\X \cap W^c})^{-1} \chi_W  =  \chi_{W}  \sum_{n=0}^\infty K^{\X \cap W^c}  (\chi_{W^c} K^{\X \cap W^c}  )^{n}\chi_{W} , 
\end{align*}
where $K^{\X \cap W^c} = K^{p_1, \cdots, p_\ell}$ is given by the formula \eqref{def-Palm-kernel} for $\X \cap W^c  = \{p_1, \cdots, p_\ell\}$. 

  Note that while in \cite[Proposition 2.5]{BQS16}, we  only treated the conditional measures with respect to fixing the configuration on   relatively compact subsets,  the same proof applies to the case when the point process is supported on the set of finite configurations.

\subsection{Monotone coupling}\label{sec-mc}

By Lyons \cite[Theorem 3.8]{Lyons-ICM}, the elementary operator-order inequality $K_\omega^{\mathfrak{p}} \le K_\omega$ implies that the reduced Palm measure $\PP_{K_\omega}^{\mathfrak{p}}  = \PP_{K_\omega^{\mathfrak{p}}}$ is {\it stochastically dominated} by $\PP_{K_\omega}$, that is, for any bounded  Borel function $f:  \Conf(U) \rightarrow \R$  such that $f(\X) \le f(\mathfrak{Y})$ whenever $\X \subset \mathfrak{Y}$, the following inequality holds
\[
\int\limits_{\Conf(U)} f d \PP_{K_\omega}^{\mathfrak{p}} \le \int\limits_{\Conf(U)} f d \PP_{K_\omega}. 
\]
 Therefore,  by Strassen's theorem \cite{Strassen-coupling} (see  Lindvall \cite{Lindvall} for a simple proof), there  exists a  monotone coupling of $\PP_{K_\omega}^{\mathfrak{p}}$ and $\PP_{K_\omega}$ in the following sense:  there exists a Borel probability measure $\nu$ on $\Conf(U) \times \Conf(U)$ whose coordinate projections $\nu_1, \nu_2$ are $\PP_{K_\omega}$ and $\PP_{K_\omega}^{\mathfrak{p}}$ respectively, such that 
\begin{align*}
\nu\Big(\Big\{(\X,  \mathfrak{Y})\in \Conf(U)\times \Conf(U)| \mathfrak{Y}\subset \X\Big\} \Big)  =1.
\end{align*}
Equivalently, there exist two coupled random configurations $\mathscr{X}$ and $\mathscr{Y}$ on $U$ defined on  a common probability space $(\Omega, \mathcal{B}, \bold{P})$,   such that 
\begin{align*}
\mathcal{L}(\mathscr{X}) = \PP_{K_\omega}, \, \mathcal{L}(\mathscr{Y}) = \PP_{K_\omega}^{\mathfrak{p}}\an \mathscr{Y}\subset \mathscr{X} \quad \text{$\bold{P}$-almost surely,}
\end{align*}
where  $\mathcal{L} ( \cdot )$ denotes the probability law of the corresponding random configuration. 

The monotone coupling between  $\PP_{\omega}^{\mathfrak{p}}$ and $\PP_{K_\omega}$  is not given explicitly (see Lyons \cite[Question 2.8]{Lyons-ICM}).
 In Lemma \ref{lem-re-f} below,  we  show that for any monotone coupling of $\PP_{K_\omega}^{\mathfrak{p}}$ and $\PP_{K_\omega}$ described as above, the conditional measure 
$$
\mathcal{L}(\mathscr{X}\setminus \mathscr{Y}| \mathscr{Y} = \mathfrak{Y})
$$
is diffusive, in the sense of Definition \ref{def-nice},   for $\PP_{K_\omega}^{\mathfrak{p}}$-almost every $\mathfrak{Y}$.

\section{Proof of the relation $\PP_{K_\omega}^{\mathfrak{p}} \ll \PP_{K_\omega}$}\label{sec-del}
\subsection{Deletion tolerance}\label{sec-del-tl}

Fix any $\ell \in \N$ and any $\ell$-tuple  $\mathfrak{p} = (p_1, \cdots, p_\ell) \in U^\ell$ of  distinct points  in $U$. We shall prove the relation $\PP_{K_\omega}^{\mathfrak{p}} \ll \PP_{K_\omega}$ by establishing, for our point
process $\PP_{K_\omega}$, the property of deletion tolerance whose definition we now recall.

\begin{definition}[Holroyd and Soo \cite{HolSoo}]\label{def-del}
Let $E$ be a locally compact Polish space. A point process $\PP$ on $E$ is called {\bf deletion tolerant},  if  for any relatively compact  Borel subset  $B \subset E$ and $\PP$-almost every configuration $\X \in \Conf(E)$, we have 
$
\PP(\#_B = 0|\X, B^c)> 0. 
$
\end{definition}

\begin{lemma}\label{lem-del}
The determinantal point process $\PP_{K_\omega}$ is deletion tolerant.
\end{lemma}

The following proposition shows that for a deletion tolerant determinantal measure, all its reduced Palm measures are indeed absolutely continuous with respect to the original determinantal measure. 

\begin{proposition}\label{prop-del-ab}
Let $\PP_K$ be a determinantal point process on a Polish space $E$ induced by a self-adjoint kernel $K$. If $\PP_K$ is deletion tolerant, then any reduced Palm measure of $\PP_K$ is absolutely continuous with respect to $\PP_K$. 
\end{proposition}
The proof of Proposition \ref{prop-del-ab} using monotone  coupling is given in  \S \ref{palm-abs}.
We proceed to the proof of Lemma \ref{lem-del}, which  uses the $H^{\infty}(U)$-module structure 
of $A^2(U, \omega)$.

\subsection{Deletion tolerance for $\PP_{K_\omega}$: proof of Lemma \ref{lem-del}}\label{sec-d-tl}
\subsubsection{Contractivity of the conditional operator}
 In our determinantal setting, the deletion tolerance has the following equivalent reformulation.  Let  $H\subset L^2(E, \mu)$ be a reproducing kernel Hilbert space such that its reproducing kernel $K$ is  locally trace class. Recall that by \eqref{cond-m-f},  for any relatively compact Borel subset $B\subset E$ and $\PP_{K}$-almost every configuration $\X$,  the conditional measure $\PP_{K}(\cdot|\X, B^c)$ is determinantal and induced by a trace class positive contraction $K^{[\X, B^{c}]}$. By the identity \eqref{id-gap} for the gap probability,  for $\PP_{K}$-almost every configuration $\X$, we have 
\[
\PP_{K}(\#_B = 0|\X, B^c) = \PP_{K^{[\X, B^c]}} (\#_B = 0) = \det (1 - K^{[\X, B^{c}]}). 
\]
 Therefore, $\PP_K$ is deletion tolerant if and only if for any relatively compact Borel subset $B\subset E$ and $\PP_{K}$-almost every configuration $\X \in \Conf(E)$, the conditional operator  $K^{[\X,B^{c}]}$ is strictly contractive. 

Our key step in the proof of   Lemma \ref{lem-del} is therefore to  prove that  for any  relatively compact Borel subset $B\subset U$ and $\PP_{K_\omega}$-almost every  configuration $\X\in \Conf(U)$, the space 
\[
\{h\in L^{2}(B): K_\omega^{[\X, B^{c}]}(h)=h\}
\]
is an $H^\infty(U)$-module and the strict contractivity of the conditional operator $K_\omega^{[\X, B^{c}]}$ follows.

\subsubsection{The module structure}

To clarify the r{\^o}le of the module structure, we formulate our results in an abstract and more general setting.  Throughout this section, we consider  a locally trace class operator $K$ that it is  the reproducing kernel of a  Hilbert space  $H\subset L^2(E, \mu)$. We need  the following additional 
\begin{assumption}\label{ass-K}
If a subset $B\subset E$ is relatively compact, then
the operator $\chi_B K$ is strictly contractive. 
\end{assumption} 

Recall  that by a sub-algebra $\mathcal{A}$ of $L^\infty(E, \mu)$, we mean a linear subspace in $L^\infty(E, \mu)$, not necessarily closed in the norm topology,  such that for any $a, b \in \mathcal{A}$, we have $a \cdot b \in \mathcal{A}$. Here $a \cdot b$ is the function obtained by pointwise multiplication of the functions $a$ and $b$.  Recall also that a linear subspace $H \subset L^2(E, \mu)$ is called an $\mathcal{A}$-module if for any $h \in H$ and $a \in \mathcal{A}$, we have $a \cdot h \in H$.

For any relatively compact Borel subset $B\subset E$ and $\PP_{K}$-almost every configuration $\X$, we have a trace class positive contraction $K^{[\X, B^{c}]}$ governing the conditional measure $\PP_K(\cdot| \X, B^c)$, cf.
 \S \ref{sec-cond-rep}. 
 Set
$$
V_{1}(K^{[\X,B^{c}]})=\{h\in L^{2}(B): K^{[\X, B^{c}]}(h)=h\}.
$$
\begin{lemma}\label{lem-module}
Let  $H\subset L^2(E, \mu)$ be a reproducing kernel Hilbert space such that its reproducing kernel $K$ is  locally trace class and satisfies Assumption \ref{ass-K}.  If $H$ is a module over a sub-algebra $\mathcal{A}\subset L^\infty(E,\mu)$,   then for any relatively compact  $B\subset E$ and $\PP_K$-almost every configuration $\X\in \Conf(E)$, the linear space $V_{1}(K^{[\X,B^{c}]})$ is an $\mathcal{A}$-module.
\end{lemma}

We can say more if the algebra $\mathcal{A}$ in Lemma \ref{lem-module} satisfies the following Assumption \ref{ass-H}.  Recall that a function $f: E\rightarrow \C$ is called {\it elementary step function} if there exists a finite Borel partition $E = \bigsqcup_{k=1}^n E_k$ and $(\lambda_1, \cdots, \lambda_n) \in \C^n$ such that $f = \sum_{k=1}^n \lambda_k \chi_{E_k}$.  

\begin{assumption}\label{ass-H}
  The algebra $\mathcal{A}$ satisfies: for any subset $C\subset E$ with $\mu(C)>0$, there exists an element $a \in \mathcal{A}$ such that the restriction $a|_{C}$ is not an elementary step function. 
\end{assumption}

\begin{lemma}\label{lem-strict-c}
Let  $H\subset L^2(E, \mu)$ be a reproducing kernel Hilbert space whose reproducing kernel $K$ is  locally trace class and satisfies Assumption \ref{ass-K}.  If $H$ is a module over a sub-algebra $\mathcal{A}\subset L^\infty(E,\mu)$ satisfying Assumption \ref{ass-H}, then for any relatively compact  $B\subset E$ and $\PP_K$-almost every $\X\in \Conf(E)$, the conditional operator $K^{[\X,B^{c}]}$ is strictly contractive, and, consequently, the determinantal point process $\PP_K$ is deletion tolerant. 
\end{lemma}

To apply Lemma \ref{lem-module} and Lemma \ref{lem-strict-c} in our concrete case of weighted Bergman kernel $K_\omega$, we still need to prove that $K_\omega$ satisfies Assumption \ref{ass-K}. This  is a  simple consequence of the following {\it unique extension property} of the space $A^2(U, \omega)$.
We say (cf. \cite[Assumption 2]{Buf-inf-1}, \cite[Assumption (A3)]{Qiu-Adv}) 
 that the  Hilbert space $H \subset L^2(E, \mu)$ has the unique extension property, if  for any $C\subset E$ with $\mu(C) >0$, we have: 
\[
\text{for all $\varphi \in H$, the condition $\varphi|_{C} \equiv 0$ implies $\varphi \equiv 0$.}
\]

\begin{lemma}\label{lem-ext}
If a reproducing kernel Hilbert space $H$  has the unique extension property, then its reproducing kernel $K$ satisfies Assumption \ref{ass-K}. 
\end{lemma}

The proofs of Lemmata  \ref{lem-module},  \ref{lem-strict-c},  \ref{lem-ext} are given  in  \S \ref{sec-lem-module}, \S \ref{sec-strict-c},   \S \ref{sec-extension} respectively. We now  conclude the proof of Lemma \ref{lem-del}. The space $A^2(U, \omega)$ is an $H^\infty(U)$-module.  Since $U$ is connected,   the Hilbert space $A^2(U, \omega)$ has the unique extension property. Then by Lemma \ref{lem-ext}, the weighted Bergman kernel $K_\omega$ satisfies Assumption \ref{ass-K}. Since $H^\infty(U)$ contains a non-constant function,  the algebra $H^\infty(U)$ satisfies Assumption \ref{ass-H}. Lemma  \ref{lem-strict-c} now gives the desired deletion tolerance of $\PP_{K_\omega}$. \qed

\subsubsection{Proof of Lemma \ref{lem-module}}\label{sec-lem-module}

Let  $H \subset L^2(E, \mu)$ be a reproducing kernel Hilbert space with a locally trace class reproducing kernel $K$ satisfying Assumption \ref{ass-K} and let $\mathcal{A} \subset L^\infty(E, \mu)$ be an algebra. Assume that $H$ is an $\mathcal{A}$-module.  
\begin{lemma}\label{lem-closed}
Let $C\subset E$.  Assume that the operator $\chi_C K$ is strictly contractive. Then for any closed subspace $S\subset H$, the linear subspace $\chi_{E\setminus C} S$ is closed in $L^2(E, \mu)$. 
\end{lemma}
\begin{proof} We use  Assumption \ref{ass-K}.
For any $h\in H$, since $h = K(h)$,  we have
\[
\| \chi_C h\|_2 = \| \chi_C K(h)\|_2 \le \|\chi_C K\|\cdot \| h\|_2.  
\]
It follows that  
\[
\|h\|_2^2 =  \| \chi_C h\|_2^2 + \| \chi_{E\setminus C} h\|_2^2  \le  \|\chi_C K\|^2 \cdot \| h\|_2^2 + \| \chi_{E\setminus C} h\|_2^2
\]
and hence 
\[
\| \chi_{E\setminus C} h\|_2 \ge  \sqrt{1 -\| \chi_C K \|^2}  \cdot \| h\|_2.
\]
It follows that the restriction map $H \rightarrow  \chi_{E\setminus C} H$,  that sends $h\in H$ to $ \chi_{E\setminus C} h$,  is a linear isomorphism between normed spaces. In particular, any closed subspace $S\subset H$ has a closed image $\chi_{E\setminus C}S$. The proof is complete. 
\end{proof}
We  start the proof of Lemma \ref{lem-module} with two elementary geometric observations. 
\begin{lemma}[First geometric observation]\label{1-geo}
Let $\widehat{P}, P_1, P_2, \cdots$ be orthogonal projections on a Hilbert space $\mathcal{H}$.  If  a vector $h\in \mathcal{H}$ satisfies
\[
h = \lim\limits_{j\to\infty} \widehat{P} P_j h,
\]
where the convergence takes place in the norm topology, then $h = \lim\limits_{j\to\infty} P_j h$. 
\end{lemma}
\begin{proof}
We have 
\begin{align}\label{norm-2-s}
\| \widehat{P}P_j h\| \le  \|  P_j h\| \le  \| h\|. 
\end{align}
The assumption $h = \lim\limits_{j\to\infty} \widehat{P} P_j h$ implies that $\| h\| = \lim\limits_{j\to\infty} \|\widehat{P} P_j h\|$, which combined with \eqref{norm-2-s} implies   $\| h\| = \lim\limits_{j\to\infty} \| P_j h\|$. Since $h - P_j h$ and $P_j h$ are orthogonal, we have
\[
\| h - P_j h\| = \sqrt{\| h\|^2 - \| P_j h\|^2},
\]
and we obtain  the desired equality
$
\lim\limits_{j\to\infty} \| h - P_j h\| = 0. 
$
\end{proof}

\begin{lemma}[Second geometric observation]\label{2-geo}
Let $P_1, P_2, \cdots$ be orthogonal projections on a Hilbert space $\mathcal{H}$ with range $\mathcal{H}_j \subset \mathcal{H}$.  Then $h  = \lim\limits_{j\to\infty} P_j h$ if and only if there exists a sequence $(h_j)_j$ of vectors with $h_j \in \mathcal{H}_j$ such that $h = \lim\limits_{j\to\infty} h_j$. 
\end{lemma}
\begin{proof}
If $h  = \lim\limits_{j\to\infty} P_j h$, then by taking $h_j = P_j h\in \mathcal{H}_j$, we obtain $h = \lim\limits_{j\to\infty} h_j$.  Conversely, assume that $h = \lim\limits_{j\to\infty} h_j$ with $h_j \in \mathcal{H}_j$, then 
\[
\| h - P_j h \|  = \min_{v \in \mathcal{H}_j}\| h - v\| \le  \| h  - h_j\| \xrightarrow{j\to\infty}0.
\]
Hence we have $h  = \lim\limits_{j\to\infty} P_j h$. 
\end{proof}

\begin{proof}[Proof of Lemma \ref{lem-module}]
We use the notation introduced in  \S \ref{sec-cm}.  By Lemma \ref{lem-closed},  the linear subspace $\chi_{B\cup F_n}H(\mathfrak{X}\cap E_n)$ is closed in $H$. Therefore,  for $\PP_{K}$-almost every configuration $\mathfrak{X}\in \Conf(E)$, the range of the operator $K^{[\X, E_n]}$ is 
  \begin{align}\label{range-id}
  \Ran  (K^{[\X, E_n]})  = \overline{\chi_{B\cup F_n}H(\mathfrak{X}\cap E_n)}^{L^2(E, \mu)} = \chi_{B\cup F_n}H(\mathfrak{X}\cap E_n).
  \end{align}

If $V_{1}(K^{[\X,B^{c}]})=\{0\}$, then it is an $\mathcal{A}$-module.

Now, we assume that  $V_{1}(K^{[\X,B^{c}]})\neq \{0\}$. Take any $h\in V_1(K^{[\X,B^{c}]})$ with $\|  h \|_{2}=1$ and any $a \in \mathcal{A}$. By definition and the convergence \eqref{operator-con}, we have 
\begin{align}\label{limit}
h=K^{[\X,B^{c}]}(h)=\lim\limits_{n\to \infty} \chi_{B} K^{[\X, E_n]}\chi_{B}(h),
\end{align}
where the convergence takes place in the norm topology of $L^2(B, \mu)$.  The relation \eqref{limit} implies  $h  = \chi_B h$ and hence 
\[
 h=\lim\limits_{n\to \infty} \chi_{B} K^{[\X, E_n]}(h). 
\]
By Lemma \ref{1-geo},  the above convergence implies 
\[
h=\lim\limits_{n\to \infty}  K^{[\X, E_n]}(h). 
\] 
For any $n\in \N$, denote $\psi_{n}:= K^{[\X, E_n]}(h)$. Then by \eqref{range-id}, we have 
\[
\psi_{n} \in \chi_{B\cup F_n} H(\X\cap E_n).
\]
In other words,  there exists an element $\phi_n \in H(\X\cap E_n)$ such that   $ \psi_{n}=\chi_{B\cup F_{n}} \cdot \phi_n$. But then 
\[
h \cdot a  = (\lim\limits_{n\to\infty} \psi_n) \cdot a = \lim\limits_{n\to\infty} \psi_n  \cdot a  = \lim\limits_{n\to\infty }\chi_{B\cup F_{n}} \cdot \phi_n \cdot a.
\]
Since $\phi_n \cdot a \in H(\X\cap E_n)$, we have $ \chi_{B\cup F_{n}} \cdot \phi_n \cdot a
 \in   \chi_{B\cup F_n} H(\X\cap E_n)$.  Therefore, by Lemma \ref{2-geo}, we obtain 
\[
h \cdot a = \lim\limits_{n\to\infty}K^{[\X, E_n]}(h \cdot a). 
\]
The above limit relation combined with the equality $h = \chi_B h$ implies 
\[
h \cdot a = \lim\limits_{n\to\infty} \chi_B K^{[\X, E_n]} \chi_B (h \cdot a) = K^{[\X, B^c]} (h \cdot a)
\]
and hence $h\cdot a \in V_1(K^{[\X, B^c]})$. The proof is complete.
\end{proof}

\subsubsection{Proof of Lemma \ref{lem-strict-c}}\label{sec-strict-c}
Let  $H \subset L^2(E, \mu)$ be a reproducing kernel Hilbert space with a locally trace class reproducing kernel $K$ satisfying Assumption \ref{ass-K} and let $\mathcal{A} \subset L^\infty(E, \mu)$ be an algebra satisfying Assumption \ref{ass-H}. Assume that $H$ is an $\mathcal{A}$-module.  

 For any Borel subset $C\subset E$, we denote 
\[
\mathcal{A}_{C} = \{   a|_C: a \in \mathcal{A}\}.
\]
 
\begin{lemma} 
If $\mu(C) >0$, then $\dim  \mathcal{A}_C  = \infty.$
\end{lemma}

\begin{proof}
By Assumption \ref{ass-H}, there exists $a\in \mathcal{A}$ such that $a|_C$ is not an elementary step function. Since $\mathcal{A}$ is an algebra, for any $n\in \N$, we have $a^n\in \mathcal{A}$. Now it suffices to show that the elements $a^n|_C$ are linearly independent.  Indeed, if this is not the case, then there exists a number $N$ and $(\lambda_1, \cdots,\lambda_N) \in \C^N\setminus \{0\}$ such that  
\[
0 = \sum_{k=1}^N   \lambda_k a^k|_C =  \sum_{k=1}^N   \lambda_k (a|_C)^k.
\]
It follows that the values of $a|_C$ is contained in the finite set of zeros of the polynomial $p(z) = \sum_{k=1}^N   \lambda_k z^k$. Therefore $a|_C$ is an elementary step function, which contradicts our assumption.
\end{proof}

\begin{corollary}\label{cor-inf}
Let  $B\subset E$ be any relatively compact Borel subset. If  $h \in H$ and $\| \chi_B h\|_2 \ne 0$, then
\[
\dim  ( h \mathcal{A}_B)  = \dim (\{  \chi_B h a: a \in \mathcal{A}\}) =  \infty.
\]
\end{corollary}

\begin{proof}
Note that the inequality $\| \chi_B h\|_2 \ne 0$ implies that there exists $C \subset B$ such that  $\mu(C)> 0$ and $|h| >0$ almost everywhere on $C$. Using the fact that the map  $ h \mathcal{A}_B \rightarrow h \mathcal{A}_C$ defined by taking the restriction on $C$ is surjective, and the fact that the map $\mathcal{A}_C \rightarrow h \mathcal{A}_C$ is bijective, we obtain 
\[
\dim  ( h \mathcal{A}_B)  \ge \dim( h \mathcal{A}_C)  =  \dim (\mathcal{A}_C) = \infty. 
\]
\end{proof}

\begin{proof}[Proof of Lemma \ref{lem-strict-c}]
By the description of the conditional operator $K^{[\X,B^c]}$, we know that 
$\| K^{[\X,B^c]} \| \leq 1$ for $\PP_{K}$-almost every configuration $\X\in {\rm Conf}(E)$.  Assume by contradiction  that $\| K^{[\X,B^c]} \| = 1$ for all configurations $\X \in \mathcal{B} \subset \Conf(E)$ with $\PP_K(\mathcal{B})>0$. Then for any $\X \in \mathcal{B}$,  since $K^{[\X,B^c]}$ is positive and compact, the spectral theorem implies that  
\[
V_1(K^{[\X,B^c]})\neq \{0\},
\]
By Lemma \ref{lem-module},  $V_1(K^{[\X,B^c]})$ is an $\mathcal{A}$-module. Take any $h \in V_1(K^{[\X,B^c]}) \setminus \{0\}$, the linear space $V_1(K^{[\X,B^c]})$ contains the linear subspace  $h \mathcal{A}_B$. By Corollary \ref{cor-inf}, we get 
\[
\dim V_1(K^{[\X,B^c]}) \ge \dim  (h \mathcal{A}_B) = \infty. 
\]
 That is, the operator $K^{[\X,B^c]}$ admits an infinite dimensional eigen-space corresponding to the eigenvalue $1$. This contradicts the fact that  the conditional operator $K^{[\X,B^c]}$ is compact for $\PP_{K}$-almost every configuration $\X\in {\rm Conf}(E)$.  
\end{proof}

\subsubsection{Proof of Lemma \ref{lem-ext}}\label{sec-extension}

First of all, by Corollary \ref{cor-inf}, we have $\dim(H) = \infty$. In particular, since the reproducing kernel $K$ of $H$ is assumed to be locally trace class, our Polish space $E$ must be non-compact.

Let $B\subset E$ be any relatively compact  Borel subset.   Assume by contradiction that $\|\chi_B K\| =1$. In particular, we must have $\mu(B) > 0$. Note also
\begin{align}\label{n-1}
\| \chi_B K \chi_B \| = \| (\chi_B K ) (\chi_B K)^*\| = \| \chi_B K\|^2 = 1.
\end{align}
Since $K$ is locally trace class, the operator $\chi_B K \chi_B$ is a positive compact operator. Therefore, the equality \eqref{n-1} implies that there exists a non-zero function  $\varphi \in L^2(B)$ such that 
\begin{align}\label{varphi-id}
\chi_B K(\chi_B \varphi)  = \chi_B K(\varphi)= \varphi. 
\end{align}
Denote $\psi : = K(\varphi)$. 
By \eqref{varphi-id}, we have $\varphi =  \chi_B K(\varphi) = \chi_B \psi$. 
Therefore, \eqref{varphi-id} can be rewritten as 
\begin{align}\label{re-written}
\chi_B K  (\chi_B \psi) = \chi_B \psi. 
\end{align}
But $\psi = K(\varphi) \in H$ implies that $K(\psi)  = \psi$ and hence 
\begin{align}\label{psi-id}
\chi_B K(\psi)  = \chi_B \psi.
\end{align}
Combining \eqref{re-written} and \eqref{psi-id}, we obtain 
\begin{align}\label{van-B}
\chi_B K(\chi_{E\setminus B} \psi) = 0. 
\end{align}
Recall that  $\mu(B)>0$, hence the equality \eqref{van-B} combined with the unique extension property of $H$ implies that
\[
K(\chi_{E\setminus B} \psi) =0,
\] 
which is equivalent to $\chi_{E\setminus B} \psi \perp H$. Therefore, we have 
\begin{align}\label{0-prod}
0  = \langle \psi, \chi_{E\setminus B} \psi \rangle  = \int_{E\setminus B} |\psi|^2 d\mu
\end{align}
and hence $\chi_{E\setminus B} \psi =0$. But since $E$ is non-compact and $B$ is relatively compact and the support of $\mu$ is the whole space $E$, we must have $\mu(E\setminus \overline{B}) > 0$.  The unique extension property of $H$ now implies $\psi  = 0$ and hence $\varphi =\chi_B\psi =0$. This contradicts the non-zero assumption on $\varphi$. 

\subsubsection{Module structure of the range of $K^{[\X, B^c]}$}
The following result will be used in the proof of insertion tolerance below. Let
$
 \overline{\Ran(K^{[\X, B^c]})}
$ be
the closure  in $L^2(E, \mu)$ of the range $\Ran(K^{[\X, B^c]})$ of the operator $K^{[\X, B^c]}$. 

\begin{lemma}\label{lem-mod-range} 
Let  $H\subset L^2(E, \mu)$ be a reproducing kernel Hilbert space such that its reproducing kernel $K$ is  locally trace class and satisfies Assumption \ref{ass-K}.  Assume that $H$ is a module over a sub-algebra $\mathcal{A}\subset L^\infty(E,\mu)$. Then for any relatively compact Borel subset $B\subset E$ and $\PP_{K}$-almost every configuration $\X\in {\rm Conf}(E)$,   the space 
 $\overline{\Ran(K^{[\X, B^c]})}$ is an $\mathcal{A}$-module. 
\end{lemma}

\begin{proof}
Since the operator $K^{[\X, B^c]}$ is self-adjoint, we have 
\begin{align}\label{ran-ker}
\overline{\Ran(K^{[\X, B^c]})} =  (\Ker  (K^{[\X, B^c]}) )^\perp, 
\end{align}
where $(\Ker  (K^{[\X, B^c]}) )^\perp$ denotes the orthogonal complement of the null-space of the operator $K^{[\X, B^c]}$, considered here as an operator acting on  $L^2(B, \mu)$.  Therefore, to prove Lemma \ref{lem-mod-range}, it suffices to prove that $\Ker  (K^{[\X, B^c]})$ is an $\bar{\mathcal{A}}$-module. Here $\bar{\mathcal{A}}$ is the complex conjugation of $\mathcal{A}$, that is, $
\bar{\mathcal{A}}: = \{\bar{a}: a \in \mathcal{A} \}$.  Indeed,  if $\Ker  (K^{[\X, B^c]})$ is an $\bar{\mathcal{A}}$-module, then given any $\varphi \in \Ker  (K^{[\X, B^c]})$ and $a \in \mathcal{A}$, we have $\bar{a} \cdot \varphi \in \Ker  (K^{[\X, B^c]})$. Therefore, for any $\psi \in \overline{\Ran(K^{[\X, B^c]})}$ and any $a \in \mathcal{A}$, we have 
\begin{align}\label{verify-ran-ker}
 \forall \varphi   \in  \Ker  (K^{[\X, B^c]}), \quad \langle a   \psi,  \varphi \rangle_{L^2(E, \mu)} = \langle  \psi, \bar{a} \varphi \rangle_{L^2(E, \mu)}  =  0. 
\end{align}
By \eqref{ran-ker}, the relation \eqref{verify-ran-ker} implies the desired relation: $a \psi \in  \overline{\Ran(K^{[\X, B^c]})}$.

\medskip

{\bf Claim A:} If $\varphi \in \Ker(K^{[\X, B^c]}) \subset L^2(B, \mu)$, then 
\begin{align}\label{claim-0}
\lim\limits_{n\to\infty} \|   K^{[\X, E_n]}  \varphi \|_2 = 0.
\end{align}
Indeed, by the following order-inequality of positive operators:
\begin{align*}
\chi_B K^{[\X, E_n]}   \chi_{F_n}  K^{[\X, E_n]}  \chi_B \le \chi_B K^{[\X, E_n]}    \chi_B,
\end{align*}
we have
\begin{multline*}
\|  \chi_{F_n}  K^{[\X, E_n]}  \varphi \|_2^2   =  \|      \chi_{F_n}  K^{[\X, E_n]}  \chi_B \varphi  \|_2^2   =
\\
 = \langle    \chi_B K^{[\X, E_n]}   \chi_{F_n}  K^{[\X, E_n]}  \chi_B \varphi, \varphi \rangle \le
\\
  \le  \langle    \chi_B K^{[\X, E_n]}   \chi_B \varphi, \varphi \rangle  
\le \| \chi_B K^{[\X, E_n]}  \varphi \|_2 \cdot \|  \varphi\|_2 .
\end{multline*}
Consequently, by noting that 
\[
K^{[\X, E_n]}  \varphi  =\chi_{F_n} K^{[\X, E_n]}  \varphi  +\chi_B K^{[\X, E_n]}  \varphi,
\]
we have 
\begin{multline}\label{nice-control}
\|  K^{[\X, E_n]}  \varphi  \|_2^2= \|\chi_{F_n} K^{[\X, E_n]}  \varphi\|_2^2  +\|\chi_B K^{[\X, E_n]}  \varphi\|_2^2 \le
\\
 \le \| \chi_B K^{[\X, E_n]}  \varphi \|_2 \cdot \|  \varphi\|_2 + \|\chi_B K^{[\X, E_n]}  \varphi\|_2^2. 
\end{multline}
But if $\varphi \in \Ker(K^{[\X, B^c]})$, then 
\begin{align}\label{ker-f}
0=  K^{[\X, B^c]}( \varphi)  = \lim\limits_{n\to\infty} \chi_B K^{[\X, E_n]}\chi_B (\varphi) =  \lim\limits_{n\to\infty} \chi_B K^{[\X, E_n]}(\varphi),
\end{align}
where the limit takes place in $L^2(B, \mu)$. Combining \eqref{nice-control} and \eqref{ker-f}, we obtain the desired relation \eqref{claim-0}.

\medskip

{\bf Claim B:}  For any $a \in \mathcal{A}$, we have the following equality of operators on $L^2(E, \mu)$: 
\begin{align}\label{eq-op}
K^{[\X, E_n]}  \bar{a}  =  K^{[\X, E_n]}  \bar{a}  K^{[\X, E_n]},  
\end{align}
where $\bar{a}$ stands for the operator $M_{\bar{a}}$ of multiplication by the bounded function $\bar{a}$.  Indeed, since the range $\chi_{B\cup F_n}H(\X\cap E_n)$ of the orthogonal projection  $K^{[\X, E_n]}$ is by definition an $\mathcal{A}$-module, we have 
\[
a K^{[\X, E_n]}  = K^{[\X, E_n]}  a K^{[\X, E_n]}. 
\]
Now by using the elemetary identity
\[
(M_{\bar{a}})^{*} = M_{a},
\]
where $*$ denotes the operation of taking the operator-adjoint, we obtain the desired equality
\[
K^{[\X, E_n]}  \bar{a}  = (M_a K^{[\X, E_n]} )^{*} = ( K^{[\X, E_n]}  M_a K^{[\X, E_n]})^{*} =  K^{[\X, E_n]}  \bar{a}  K^{[\X, E_n]}. 
\]

Now take any $\varphi \in \Ker  (K^{[\X, B^c]}) \subset L^2(B, \mu)$ and any $a \in \mathcal{A}$. By applying \eqref{claim-0} and  the contractivity of the operator $\chi_B   K^{[\X, E_n]}$,  we have 
\begin{align*}
\| K^{[\X, B^c]} (\bar{a}\varphi) \|_2 &  = \lim\limits_{n\to\infty} \| \chi_B K^{[\X, E_n]}\chi_B ( \bar{a} \varphi) \|_2 & 
\\
& = \lim\limits_{n\to\infty} \| \chi_B K^{[\X, E_n]} ( \bar{a} \varphi) \|_2 & 
\\
&  = \lim\limits_{n\to\infty} \| \chi_B   K^{[\X, E_n]}  \bar{a}  K^{[\X, E_n]} (\varphi)\|_2 \quad   & (\text{by the equality \eqref{eq-op}})
\\
& \le \limsup_{n\to\infty}   \|a \|_\infty \cdot \|  K^{[\X, E_n]} (\varphi)\|_2   & \quad   (\text{by contractivity of $K^{[\X, E_n]}$})
\\
&=0. & \quad (\text{by the equality \eqref{claim-0}})
\end{align*}
Hence $\bar{a} \cdot \varphi \in \Ker(K^{[\X, B^c]})$ and this completes the proof of Lemma \ref{lem-mod-range}. 
\end{proof}

Lemma \ref{lem-mod-range} implies that the space $\overline{\Ran(K^{[\X, B^c]})}$ is either the null space $\{0\}$ or of infinite dimension.  Thus we obtain
\begin{corollary}\label{cor-rank}
If $K^{[\X, B^c]} \ne 0$, then $\rank (K^{[\X, B^c]}) = \infty$. 
\end{corollary}

\subsection{Proof of Proposition \ref{prop-del-ab}} \label{palm-abs}

Let $\PP_K$ be a determinantal point process  on a Polish space $E$ induced by a locally trace class positive contraction $K$ on $L^2(E, \mu)$. Then for $K(x, x) d\mu$-almost every $p \in E$, the reduced Palm measure $\PP_K^{p}$ of $\PP_K$ corresponding to the position $p$ is the determinantal point process induced by the kernel 
\[
K^p(x, y) = K(x, y) - \frac{K(x, p) K(p, y)}{K(p, p)}. 
\]

The operator-order inequality $K^p \le K$ then follows immediately. By iteration, for $\det(K(x_i, x_j)) d\mu^{\otimes \ell}$-almost every $ \mathfrak{p} = (p_1, \cdots, p_\ell)$, we obtain the operator-order inequality $K^{\mathfrak{p}} \le K$ and in particular
\begin{align}\label{p-diag}
K(x, x) - K^{\mathfrak{p}}(x, x) \ge 0.
\end{align}
  The operator-order inequality $K^{\mathfrak{p}} \le K$ implies that there exists a monotone coupling between $\PP_K^{\mathfrak{p}}$ and $\PP_K$, cf. \S \ref{sec-mc}. Note that the positive operator $K- K^{\mathfrak{p}}$ has finite rank and finite trace:
\begin{align}\label{f-trace}
\tr(K-K^{\mathfrak{p}})< \infty.
\end{align}

We shall use an equivalent formulation of deletion tolerance due to  Holroyd and Soo. In  \cite[Theorem 1.1]{HolSoo} Holroyd and Soo showed that 
a point process $\PP$ on $E$ is deletion tolerant if and only if for any two coupled random configurations $\mathscr{X}$  and $\mathscr{Z}$ on $E$ defined on a common probability space $(\Omega, \mathcal{B}, \bold{P})$ and satisfying
\[
\text{$\mathcal{L} (\mathscr{X}) = \PP$  and  $\mathscr{Z}\subset \mathscr{X}$ \text{$\bold{P}$-almost surely} and $\#\mathscr{Z} < \infty$ \text{$\bold{P}$-almost surely}, }
\]  
we have
\begin{align}\label{f-del}
\mathcal{L} ( \mathscr{X} \setminus \mathscr{Z}) \ll \mathcal{L} ( \mathscr{X}).
\end{align}

\begin{lemma}\label{lem-Claim-A}
For two coupled random configurations $\mathscr{X}$ and $\mathscr{Y}$ on $E$ defined on the same probability space  $(\Omega, \mathcal{B}, \bold{P})$ and satisfying 
\begin{align*}
\mathcal{L}(\mathscr{X}) = \PP_{K}, \, \mathcal{L}(\mathscr{Y}) = \PP_{K}^{\mathfrak{p}}\an \mathscr{Y}\subset \mathscr{X} \quad \text{$\bold{P}$-almost surely,}
\end{align*}
we have
\[
\# [ \mathscr{X} \setminus \mathscr{Y}]< \infty, \text{$\bold{P}$-almost surely.}
\]
\end{lemma}

\begin{proof}
Indeed, denote $\mathscr{Z}  : = \mathscr{X} \setminus \mathscr{Y}$. It suffices to show that $\E\#\mathscr{Z} < \infty$.  Fix any exhausting sequence of relatively compact subsets $B_1 \subset \cdots \subset B_n \subset \cdots \subset E$.  By \eqref{p-diag} and \eqref{f-trace},  we have 
\begin{multline*}
\E [ \#(\mathscr{X}|_{B_n})-  \#(\mathscr{Y}|_{B_n})]  = \int_{B_n} ( K(z, z) - K^{\mathfrak{p}}(z, z))  d\mu(z)\le
\\
 \le \int_{E} ( K(z, z) - K^{\mathfrak{p}}(z,z))  d\mu(z)
= \tr (K - K^{\mathfrak{p}})< \infty. 
\end{multline*}
We have a monotone increasing sequence $(\#(\mathscr{Z}|_{B_n}))_{n=1}^\infty$ which converges to $\#\mathscr{Z}$. Therefore, by monotone convergence theorem,  we get the desired inequality
\[
\E \#\mathscr{Z} = \lim\limits_{n\to\infty}  \E \#(\mathscr{Z}|_{B_n}) =   \lim\limits_{n\to\infty}  \E [ \#(\mathscr{X}|_{B_n})-  \#(\mathscr{Y}|_{B_n})] \le  \tr (K - K^{\mathfrak{p}}) < \infty. 
\]
\end{proof}

  By the existence of the monotone coupling  of $\PP_{K}^{\mathfrak{p}}$ and $\PP_{K}$, let $\mathscr{X}$ and $\mathscr{Y}$ be  two coupled random configurations on $E$ defined on the same probability space $(\Omega, \mathcal{B}, \bold{P})$,   such that 
\begin{align*}
\mathcal{L}(\mathscr{X}) = \PP_{K}, \, \mathcal{L}(\mathscr{Y}) = \PP_{K}^{\mathfrak{p}}\an \mathscr{Y}\subset \mathscr{X} \quad \text{$\bold{P}$-almost surely.}
\end{align*}
By Lemma \ref{lem-Claim-A}, the  random configuration $\mathscr{Z} = \mathscr{X}\setminus \mathscr{Y}$ is almost surely finite. Therefore, under the hypothesis  of  Proposition \ref{prop-del-ab} that the determinantal point process $\PP_{K}$ is deletion tolerant,  we can apply the relation \eqref{f-del} and obtain the desired relation: 
\[
\PP_{K}^{\mathfrak{p}} = \mathcal{L}(\mathscr{Y}) = \mathcal{L}(\mathscr{X} \setminus \mathscr{Z}) \ll \mathcal{L}(\mathscr{X})  = \PP_{K}. 
\]

\section{Proof of the relation $\PP_{K_\omega} \ll \PP_{K_\omega}^{\mathfrak{p}}$}\label{sec-dif-direction}

\subsection{Insertion tolerance}

Fix any $\ell \in \N$ and any $\ell$-tuple  $\mathfrak{p} = (p_1, \cdots, p_\ell) \in U^\ell$ of  distinct points  in $U$. We start the proof of the relation $\PP_{K_\omega} \ll \PP_{K_\omega}^{\mathfrak{p}}$ by establishing, for our 
reduced Palm measure $\PP_{K_\omega}^{\mathfrak{p}}$, the property of insertion tolerance whose definition we now recall.

Recall that given two point processes $\PP_1$ and $\PP_2$ on $U$, their convolution $\PP_1 * \PP_2$ is defined by 
\[
\PP_1*\PP_2 = \mathcal{L} ( \mathscr{X}_1 \cup \mathscr{X}_2), 
\]  
with $\mathscr{X}_1, \mathscr{X}_2$  random configurations independently sampled with $\PP_1, \PP_2$ respectively.

\begin{definition}[Holroyd and Soo \cite{HolSoo}]\label{def-ins}
A point process $\PP$ on $U$ is called insertion tolerant if for any relatively compact $S \subset U$ with non-zero Lebesgue measure, 
\[
\PP * \delta_{\Unif(S)} \ll \PP,
\]
where $\delta_{\Unif(S)}$ denotes the point process of one single random particle with uniform distribution $\Unif(S)$ on $S$. 
\end{definition}
\begin{remark}
In the above definition of insertion tolerance, it suffices to consider  open connected relatively compact subsets $S\subset U$. Indeed,  any relatively compact $S\subset U$ is contained in a relatively compact, open connected subset $\widetilde{S} \subset U$, we have $\Unif(S) \ll \Unif(\widetilde{S})$, and the relation $\PP * \delta_{\Unif(\widetilde{S})} \ll \PP$ implies $\PP * \delta_{\Unif(S)} \ll \PP$. 
\end{remark}
\begin{lemma}\label{lem-ins-palm}
The determinantal point process $\PP_{K_\omega}$ and the determinantal point process $\PP_{K_\omega}^{\mathfrak{p}}$ are both insertion tolerant.
\end{lemma}

The proof of Lemma \ref{lem-ins-palm} is postponed till \S \ref{sec-ins-tl}. We now proceed to the derivation  of $\PP_{K_\omega} \ll \PP_{K_\omega}^{\mathfrak{p}}$ from Lemma \ref{lem-ins-palm}.  We shall need an equivalent formulation of insertion tolerance due to Holroyd and Soo. 

First denote $\Conf_{f}(U)\subset \Conf(U)$ the subset of finite configurations on $U$, that is,  
\begin{align*}
\Conf_{f}(U): = \{\X \in \Conf(U)| \# \X < \infty\}.
\end{align*}
 A point process $\PP$ on $U$ is called finite, if it is supported on the set of finite configurations:  $\PP(\Conf_f(U))  =1$. The space $\Conf_f(U)$ is equipped with a $\sigma$-finite measure 
\begin{align*}
\sigma_U: = \sum_{n=0}^\infty  \sigma_n = \delta_{\emptyset} + \sum_{n=1}^\infty  \pi^{(n)}_{*}((dV|_U)^{\otimes n}),
\end{align*}
where $\sigma_0 = \delta_\emptyset$ is the Dirac measure on the empty configuration and  $\sigma_n= \pi^{(n)}_{*}((dV|_U)^{\otimes n})$ denotes the image measure of the Lebesgue measure on $U^n$ under the natural map $\pi^{(n)}: U^n \rightarrow \Conf(U)$  defined by $
\pi^{(n)} ( (x_1, \cdots, x_n)) = \{x_1, \cdots, x_n\}$. Note that $\sigma_n$'s are all mutually singular. 

\begin{definition}\label{def-nice} 
A finite point process $\PP$ on $U$ is called {\bf diffusive} if $\PP \ll  \sigma_U$. 
\end{definition}

Holroyd and Soo \cite[Corollary 3.2]{HolSoo} showed that 
a point process $\PP$ on $U$ is insertioin tolerant if  and only if for any coupled random configurations $\mathscr{Y}$  and $\mathscr{Z}$ on $U$  such that 
\[
\mathcal{L}(\mathscr{Y}) = \PP \an \#\mathscr{Z} < \infty
\]
and the conditional measure $\mathcal{L}(\mathscr{Z}| \mathscr{Y} = \mathfrak{Y})$
is {\it diffusive} for $\PP$-almost every $\mathfrak{Y}$, one has
\begin{align}\label{f-ins}
\mathcal{L}( \mathscr{Y}\cup \mathscr{Z}) \ll \mathcal{L}(\mathscr{Y}) = \PP.
\end{align}

\begin{lemma}\label{lem-re-f}
Let $\mathscr{X}$ and $\mathscr{Y}$  be any two coupled random configurations on $U$ defined on  a common probability space  $(\Omega, \mathcal{B}, \bold{P})$,   such that 
\begin{align*}
\mathcal{L}(\mathscr{X}) = \PP_{K_\omega}, \, \mathcal{L}(\mathscr{Y}) = \PP_{K_\omega}^{\mathfrak{p}}\an \mathscr{Y}\subset \mathscr{X} \quad \text{$\bold{P}$-almost surely.}
\end{align*}
 Then  the conditional law $\mathcal{L}(\mathscr{X}\setminus \mathscr{Y}| \mathscr{Y} = \mathfrak{Y})$ is  diffusive for $\PP_{K_\omega}^{\mathfrak{p}}$-almost every $\mathfrak{Y}$.  
\end{lemma}

The proof of Lemma \ref{lem-re-f} is postponed till \S \ref{sec-diff} and we  continue the derivation of the relation  $\PP_{K_\omega} \ll \PP_{K_\omega}^{\mathfrak{p}}$ from Lemmata \ref{lem-ins-palm} and \ref{lem-re-f}. By the existence of the monotone coupling  of $\PP_{K_\omega}^{\mathfrak{p}}$ and $\PP_{K_\omega}$, there exist two coupled random configurations $\mathscr{X}$ and $\mathscr{Y}$ on $U$ defined on the same probability space $(\Omega, \mathcal{B}, \bold{P})$,   such that 
\begin{align*}
\mathcal{L}(\mathscr{X}) = \PP_{K_\omega}, \, \mathcal{L}(\mathscr{Y}) = \PP_{K_\omega}^{\mathfrak{p}}\an \mathscr{Y}\subset \mathscr{X} \quad \text{$\bold{P}$-almost surely.}
\end{align*}
By Lemma \ref{lem-re-f},  the conditional law 
\[
\mathcal{L}(\mathscr{X}\setminus \mathscr{Y}| \mathscr{Y} = \mathfrak{Y})
\]
is diffusive for $\PP_{K_\omega}^{\mathfrak{p}}$-almost every $\mathfrak{Y}$. By Lemma \ref{lem-ins-palm}, the reduced Palm measure $\PP_{K_\omega}^{\mathfrak{p}} = \mathcal{L}(\mathscr{Y})$ is insertion tolerant. Therefore, we can add measurably a conditionally diffusive random finite set $\mathscr{X}\setminus \mathscr{Y}$ (see Lemma \ref{lem-Claim-A} for the finiteness) to  the random configuration $\mathscr{Y}$  and obtain the desired relation: 
\[
\PP_{K_\omega}= \mathcal{L}(\mathscr{X}) = \mathcal{L}(\mathscr{Y}  \cup (\mathscr{X}\setminus \mathscr{Y}) ) \ll \mathcal{L}(\mathscr{Y})  = \PP_{K_\omega}^{\mathfrak{p}}. 
\]

\begin{remark}
Note that in the above derivation  of $\PP_{K_\omega} \ll \PP_{K_\omega}^{\mathfrak{p}}$,  we only use the insertion tolerance of $\PP_{K_\omega}^{\mathfrak{p}}$ and  do not use the insertion tolerance of $\PP_{K_\omega}$. 
\end{remark}

\subsection{Number insertion tolerance of   $\PP_{K_\omega}$ and $\PP_{K_\omega}^{\mathfrak{p}}$}

One important intermediate step for proving  the insertion tolerance of  the determinantal point processes  $\PP_{K_\omega}$ and $\PP_{K_\omega}^{\mathfrak{p}}$   is to prove their {\it number insertion tolerance}, whose definition is given as follows.   

\begin{definition}\label{def-num-ins}
Given a point process $\PP$ on $U$, we say that  $\PP$ is {\bf number insertion tolerant} if for any relatively compact $B\subset U$ with positive Lebesgue measure, we have
 \begin{align*}
\PP(\#_{B}>0|\X,B^{c})>0 \quad \text{for $\PP$-almost every $\X$.}
\end{align*}
\end{definition}

\begin{lemma}\label{lem-num-ins}
The determinantal point process $\PP_{K_\omega}$ and the determinantal point process $\PP_{K_\omega}^{\mathfrak{p}}$ are both number insertion tolerant.
\end{lemma}
If $B\subset U$ is as in Lemma \ref{lem-num-ins}, then, for $\PP_{K_\omega}$-almost every $\X$, we have 
\[
\PP_{K_\omega}(\#_{B}>0|\X,B^{c}) =\PP_{K_\omega^{[\X, B^c]}}(\#_{B}>0) =  1- \det(1- K_\omega^{[\X, B^c]}).
\]
 Lemma \ref{lem-num-ins} claims that for $\PP_{K_\omega}$-almost every $\X$, the operator $K_\omega^{[\X, B^c]}$ is a non-zero operator. The same result holds for the reproducing kernel $K_{\omega}^{\mathfrak{p}}$. 
The sections \S \ref{sec-an-kernel}, \S \ref{pf-real-an}  and \S \ref{sec-pf-lem-ins}  are devoted to the proof of Lemma \ref{lem-num-ins}.

\subsubsection{The case of real analytic kernels } \label{sec-an-kernel}
We start with an auxiliary result.

\begin{proposition}\label{prop-anal}
Let $E \subset \R^d$ be an open connected subset, equipped with a Radon measure $\mu$, absolutely continuous with respect to the Lebesgue measure on $E$ and whose support is the whole space $E$. Assume that $K$ is a { trace class} positive contraction of infinite rank  on $L^2(E, \mu)$  that 
 admits a version of kernel $K(t, s)$ which is real-analytic on $E\times E$.
Then the determinantal point process $\PP_K$ is number insertion tolerant. 
\end{proposition}
\begin{proof}
The trace class assumption $\tr(K)< \infty$ on $K$ implies, cf. \S \ref{sec-mc}, that  for any Borel subset $W\subset E$ and $\PP_K$-almost every $\X\in \Conf(E)$, the conditional measure $\PP_K(\cdot | \X, W^c)$ is a determinantal point process and is induced by the kernel 
\begin{align*}
K^{[\X, W^c]}:  = \chi_{W} K^{\X \cap W^c} (1 - \chi_{W^c} K^{\X \cap W^c})^{-1} \chi_W  =  \chi_{W}  \sum_{n=0}^\infty K^{\X \cap W^c}  (\chi_{W^c} K^{\X \cap W^c}  )^{n}\chi_{W},
\end{align*}
where $K^{\X \cap W^c}$ is given as follows:  if $\X \cap W^c = \{p_1, \cdots, p_r \}$, then,  setting $p_0= x, q_0 = y$, $q_i = p_i$ for $1\le i \le r$, we have
 \begin{align}\label{palm-kernel}
 K^{\X \cap W^c}(x, y) : =
 \displaystyle \frac{\Det[K(p_i, q_j)]_{0 \le i, j \le r}
}{
 \Det[K(p_i, p_j)]_{1\le i, j \le r}
}= K(x, y) - \sum_{i, j =1}^r \alpha_{ ij}  K(x, p_i) K(p_j, y),  
\end{align}
where the constants $ \alpha_{ ij}$'s are given by the formula
\[
 \alpha_{ ij}  = \frac{ P_{ij} (K(p_u, p_v))_{1 \le u, v \le r}
}{
 \Det[K(p_u, p_v)]_{1\le u, v \le r} 
}, \quad \text{ with $P_{ij}$'s polynomials.}
\]
In particular, $K^{\X \cap W^c}(x, y)$ is real analytic in the coordinates $(x, y)$. 

By definition,  $\PP_K$ is number insertion tolerant if and only if for any relatively compact $B \subset E$ with $\mu(B) > 0$, we have 
\[
\PP_K (\#_B > 0| \X, B^c) = \PP_{K^{[\X, B^c]}} (\#_B >0) > 0 \quad \text{for $\PP_K$-almost every $\X$.} 
\]
Therefore, to show that $\PP_K$ is number insertion tolerant is equivalent to show that for any relatively compact $B \subset E$ with $\mu(B) > 0$, 
\begin{align}\label{ins-eq}
K^{[\X, B^c]} \ne 0  \quad \text{for $\PP_K$-almost every $\X$.}
\end{align}
Assume, by contradiction with \eqref{ins-eq}, that
\begin{align}\label{0-op}
0 = K^{[\X, B^c]} =  \sum_{n=0}^\infty  \chi_{B}   K^{\X \cap B^c}  (\chi_{B^c} K^{\X \cap W^c}  )^{n}\chi_{B} \quad \text{for $\PP_K$-almost every $\X$.}
\end{align}
Since for any $n =0, 1, \cdots, $ the operators $\chi_{B}   K^{\X \cap B^c}  (\chi_{B^c} K^{\X \cap W^c}  )^{n}\chi_{B}$ are positive, the equality \eqref{0-op} implies that 
\[
\chi_{B}   K^{\X \cap B^c}  (\chi_{B^c} K^{\X \cap W^c}  )^{n}\chi_{B} = 0, \text{for any $n = 0, 1, \cdots.$}
\]
In particular, for $n=0$,  we get 
\[
 [\chi_B    (K^{\X \cap B^c})^{1/2} ]  [ \chi_B    (K^{\X \cap B^c})^{1/2}]^* =   \chi_{B}   K^{\X \cap B^c} \chi_{B}  = 0\]
 and hence $ \chi_B    (K^{\X \cap B^c})^{1/2} =0$. Then
 \[
  \chi_B    K^{\X \cap B^c} =  \chi_B    (K^{\X \cap B^c})^{1/2}     (K^{\X \cap B^c})^{1/2} =0. 
 \]
Therefore,  the real analytic function $K^{\X \cap B^c} (x, y)$ vanishes, up to a $\mu^{\otimes 2}$-negligible set, on the subset $B\times E \subset E \times E$. Since $\mu$ is absolutely continuous with respect to the Lebesgue measure, $B\times E$ is of positive Lebesgue measure.  But since $E$ is connected and $\mu(B)> 0$, we must have 
\begin{align}\label{id-0}
K^{\X \cap B^c} (x, y) =0 \,\, \text{for all $(x, y) \in E\times E$.}
\end{align}
Substituting the equality \eqref{id-0} into the equality \eqref{palm-kernel},  we get 
\[
K(x, y)  = \sum_{i, j =1}^r \alpha_{ ij}  K(x, p_i) K(p_j, y),  \quad \text{if $\X \cap B^c = \{p_1, \cdots, p_r\}.$}
\]
Hence $\rank( K) \le r^2 < \infty$, which contradicts the original assumption  $\rank(K) = \infty$. 
\end{proof}

\begin{remark}
Without the analyticity condition, Proposition \ref{prop-anal} fails. Indeed, take any $E_0\subset E$ with $\mu(E\setminus E_0)> 0$ and let $K_0$ be an infinite rank trace class positive contraction  on $L^2(E_0, \mu)$.  Using the  decomposition $L^2(E, \mu) = L^2(E_0, \mu)\oplus L^2(E\setminus E_0, \mu)$, we define
 \begin{align*}
K =  \left[\begin{array}{cc} K_0 & 0 \\ 0 &0 \end{array}\right]. 
\end{align*}
Then $K$ is an infinite rank trace class positive contraction  on $L^2(E, \mu)$. But for any relatively compact $B\subset E\setminus E_0$ with $\mu(B) > 0$, we have $\PP_K(\cdot | \X, B^c) = \delta_{\emptyset}$. This implies that $\PP_K$ is not number insertion tolerant. 
\end{remark}

\subsubsection{Real analyticity of the conditional kernel $K_\omega^{[\X, B^c]}$}\label{pf-real-an} We start with an elementary   lemma, cf. Krantz \cite[Lemma 1.1.1]{Krantz}.
\begin{lemma}\label{lem-krantz}
Let $B \subset U$ be a relatively compact open subset in $U$ and $A \subset B$ be any compact subset in $B$.  There exists a constant $C_{A, B}>0$ such that 
\[
\text{$\sup_{z\in A} | f(z)| \le C_{A, B} \| f\|_{L^2(B, \omega dV)}$,  for all $f \in A^2(B, \omega)$,}
\]
where 
\[
A^2(B, \omega): = \Big\{\text{$f: B \rightarrow\C$ holomorphic on $B$} \Big|  f \in L^2(B, \omega dV)\Big\}.
\]
In particular, $A^2(B, \omega)$ is closed in $L^2(B, \omega dV)$. 
\end{lemma}
\begin{proof}
Since $A$ is compact and $B$ is open in the open subset $U\subset \C^d$, there exists a constant $r  = r(A, B)> 0$ such that for any $z \in A,  B(z, r) \subset B$. Here $B(z, r)$ is the usual Euclidean ball with center $z$ and radius $r$.

Therefore, for each $z\in A$ and $f \in   A^2(B, \omega)$, the mean-value property for holomorphic functions implies that 
\begin{multline*}
| f(z) |   = \left|  \frac{1}{V(B(z, r))} \int_{B(z, r)} f(t) dV(t)\right|  \le
\\
 \le   \frac{1}{V(B(z, r))}  \left( \int_{B(z, r)} \frac{1}{\omega(t)} dV(t)\right)^{1/2}\left(  \int_{B(z, r)} |f(t) |^2 \omega(t) dV(t)\right)^{1/2}.
\end{multline*}
By the assumption \eqref{weight-low}, we have 
$$
| f(z)|  \le \frac{V(B(z, r))^{-1/2} }{(\displaystyle{\essinf_{z\in B} \omega(z)})^{1/2}} \| f\|_{L^2(B, \omega dV)} =  \frac{c_d r^{-d/2}}{(\displaystyle{\essinf_{z\in B} \omega(z)})^{1/2}}  \| f\|_{L^2(B, \omega dV)} . 
$$
Thus we may take 
\[
C_{A, B} = \frac{c_d r^{-d/2}}{(\displaystyle{\essinf_{z\in B} \omega(z)})^{1/2}} < \infty.
\]

Therefore, for sequences in $A^2(B,\omega)$, the convergence
in $L^2(B,\omega dV)$-norm yields uniform convergence on compacts, so the limit does belong to $A^2(B,\omega)$ as well, hence $A^2(B,\omega)$ is closed.
\end{proof}

\begin{lemma}\label{lem-real-anal}
For a relatively compact open subset $B \subset U$, the map $(z, w) \mapsto K_\omega^{[\X, B^c]}(z, \overline{w})$ is holomorphic on $B \times B$ and the map $(z, w) \mapsto K_\omega^{[\X, B^c]}(z, w)$ is real analytic on $B\times B$. 
\end{lemma}

\begin{proof}
The operator $K_\omega^{[\X, B^c]}$ is positive and has finite trace. Let $\lambda_1 \ge \cdots \ge \lambda_k \ge \cdots \ge 0$ be  the eigenvalues and $(\varphi_k)_{k=1}^\infty$  be the $L^2$-normalized eigenfunctions for $K_\omega^{[\X, B^c]}$. We have $\sum_{k=1}^\infty \lambda_k <\infty$ and 
\[
K_\omega^{[\X, B^c]} (z, w)   = \sum_{k=1}^\infty    \lambda_k \varphi_k(z) \overline{\varphi_k(w)} =  \sum_{k\in \N: \lambda_k > 0}   \lambda_k \varphi_k(z) \overline{\varphi_k(w)}.
\]
Now we show that for any $k\in \N$ with $\lambda_k>0$, the corresponding eigenfunction $\varphi_k$ is holomorphic on $B$.  Indeed, we have 
\[
\varphi_k = \frac{1}{\lambda_k} K_\omega^{[\X, B^c]}  (\varphi_k) = \frac{1}{\lambda_k}  \lim\limits_{n\to\infty} \chi_B K_\omega^{[\X, E_n]} \chi_B (\varphi_k),
\]
where the limit takes place in $L^2(B, \omega dV)$. Since 
\[
\chi_B K_\omega^{[\X, E_n]} \chi_B (\varphi_k) \in  A^2(B, \omega),
\]
and $A^2(B, \omega)$ is closed in $L^2(B, \omega dV)$, we obtain $\varphi_k \in A^2(B, \omega)$. 

Now for any $k \in \N$ with $\lambda_k >0$, the function 
\[
(z, w) \mapsto \varphi_k(z) \overline{\varphi_k(\bar{w})}
\]
 is holomorphic on $B\times B \subset \C^d \times \C^d$.  For any compact subset $A\subset B$,  by Lemma \ref{lem-krantz}, there exists $C_{A, B}> 0$, such that for any $k \in \N$ with $\lambda_k >0$, 
\[
\sup_{(z, w) \in A\times A} | \varphi_k(z) \overline{\varphi_k(\bar{w})}| \le C_{A, B} \| \varphi_k\|_{A^2(B, \omega)}^2 = C_{A, B}. 
\]
It follows that the  series 
\begin{align*}
K_\omega^{[\X, B^c]} (z, \bar{w})    =  \sum_{k\in \N: \lambda_k > 0}   \lambda_k \varphi_k(z) \overline{\varphi_k(\bar{w})}  
\end{align*}
converges uniformly on any compact subset of $B\times B$, that the map $(z, w) \mapsto K_\omega^{[\X, B^c]}(z, \overline{w})$ is holomorphic on $B \times B$ and  the map $(z, w) \mapsto K_\omega^{[\X, B^c]}(z, w)$ is real analytic on $B\times B$. 
\end{proof}

\begin{remark}\label{real-complex-rem}
For general real analytic  reproducing kernels, we do not know whether the corresponding conditional kernels are still real analytic.    The proof above that the conditional kernel $K_\omega^{[\X, B^c]}(z, w)$  is real analytic uses that  uniform convergence on compact subsets preserves  holomorphicity, a property  not shared by real analytic functions. 
\end{remark}

\subsubsection{Proof of Lemma \ref{lem-num-ins}}\label{sec-pf-lem-ins}
First fix an exhausting sequence 
\[
B_1 \subset B_2 \subset \cdots \subset B_n \subset \cdots \subset  U
\] 
of relatively compact  open connected subsets of $U$.  Let $B\subset U$ be a relatively compact subset satisfying $\mu(B) > 0$. Using \cite[Lemma 7.2]{BQS16} and  the  relation (6.7) in \cite[Proposition 6.4]{BQS16}, for $\PP_{K_\omega}$-almost every $\X$  we can choose $n$ large enough such that $
B \subset B_n$
and 
\begin{align}\label{cond-non-0}
  \PP_{K_\omega} (\#_{B_n} > 0 | \X, B_n^c) > 0. 
\end{align}
Indeed,  for $n$ large enough and such that $B \subset B_n$, we have 
\[  
\PP_{K_\omega} (\#_{B_n} > 0 | \X, B_n^c)  \ge \PP_{K_\omega} (\#_{B} > 0 | \X, B_n^c).
\]
Therefore, for $\PP_{K_\omega}$-almost every configuration $\X \in \Conf(U)$, 
\begin{multline*}
\liminf_{n\to\infty}\PP_{K_\omega} (\#_{B_n} > 0 | \X, B_n^c)    \ge \lim\limits_{n\to\infty} \PP_{K_\omega} (\#_{B} > 0 | \X, B_n^c)=
\\
  =  \PP_{K_\omega} (\#_{B} > 0) 
= 1 - \det(1 - \chi_B K_\omega\chi_B). 
\end{multline*}
Since $\chi_B K_\omega\chi_B$ is a non-zero trace class operator,  we have $\det(1 - \chi_B K_\omega\chi_B) < 1$ and hence obtain  the desired inequality \eqref{cond-non-0}. 
The inequality \eqref{cond-non-0} combined with the equality $\PP_{K_\omega}(\cdot | \X, B_n^c) = \PP_{K_\omega^{[\X, B_n^c]}}$, implies that 
$
K_\omega^{[\X, B_n^c]} \ne 0.
$
By Lemma \ref{lem-real-anal}, the function $(z, w) \mapsto K_\omega^{[\X, B_n^c]}(z, w)$  is real analytic on $B_n \times B_n \subset \C^d \times \C^d \simeq \R^{2d} \times \R^{2d}$. Corollary \ref{cor-rank} implies  
that  $K_\omega^{[\X, B_n^c]}$ has infinite rank.
Since $B_n$ is relatively compact in $U$, the operator $K_\omega^{[\X, B_n^c]}$ is trace class. 
The kernel $K_\omega^{[\X, B_n^c]}$ thus satisfies the assumptions of Proposition \ref{prop-anal}, which implies that   $\PP_{K_\omega^{[\X, B_n^c]}}$ is number insertion tolerant. In particular, we have 
\begin{align}\label{ins-tol-cond}
\PP_{K_\omega^{[\X, B_n^c]}} ( \#_B> 0 | \X, B_n\setminus B) > 0.
\end{align}
Noting  the  measure-theoretical identity 
\[
\PP_{K_\omega^{[\X, B_n^c]}} ( \cdot| \X, B_n\setminus B) =   [\PP_{K_\omega}(\cdot | \X, B_n^c) ]( \cdot| \X, B_n\setminus B)   = \PP_{K_\omega}(\cdot | \X, B^c),
\]
we obtain that the inequality \eqref{ins-tol-cond} is equivalent to the inequality
\[
\PP_{K_\omega}(\#_B >0 | \X, B^c)>0,
\]
and the proof of number insertion tolerance of  $\PP_{K_\omega}$ is complete.  
The argument also yields the number insertion tolerance for $\PP_{K_\omega}^{\mathfrak{p}}$ : simply note  that the kernel $K_\omega^{\mathfrak{p}}$ corresponding to  the reduced  Palm measure $\PP_{K_\omega}^{\mathfrak{p}}$ is the reproducing kernel of the Hilbert subspace $A^2(U, \omega; \mathfrak{p}) $ defined in \eqref{berg-palm-sp} and that 
{$A^2(U, \omega; \mathfrak{p})$ is a sub-$H^\infty(U)$-module of $A^2(U, \omega)$.}

\subsection{Insertion tolerance of $\PP_{K_\omega}$ and $\PP_{K_\omega}^{\mathfrak{p}}$}\label{sec-ins-tl}

\begin{lemma}\label{lem-nice}
Let $\PP$ be a point process on $U$. Assume that for any non-empty relatively compact connected open subset $S \subset U$  and for $\PP$-almost every $\X$,  we have 
\[
\PP(\#_S = n| \X, S^c) > 0, \quad \text{for all $n\in\N$,}
\] 
and moreover, 
\[
[\PP(\cdot | \X, S^c)](\cdot| \#_S  =n)   \simeq \pi^{(n)}_{*} ((dV|_S)^{\otimes n}), \quad \text{for all $n\in\N$,} 
\]
where $\simeq$ denotes the relation of mutual absolute continuity.  Then $\PP$ is  insertion tolerant. 
\end{lemma}

\begin{proof}
By definition, we have
\[
\Big[ \pi^{(n)}_{*} \Big((\frac{1}{V(S)}dV|_S)^{\otimes n}\Big)\Big]* \delta_{\Unif(S)} = \pi^{(n+1)}_{*} \Big((\frac{1}{V(S)}dV|_S)^{\otimes (n+1)}\Big) \quad \text{for all $n\in\N$}
\]
and 
\[
\PP(\cdot |\X, S^{c}) = \sum_{n\in\N} \PP(\#_S = n| \X, S^c) \cdot  [\PP(\cdot | \X, S^c)](\cdot| \#_S  =n),  
\]
whence 
\begin{align}\label{cond-ins} 
\PP(\cdot |\X, S^{c}) * \delta_{\Unif(S)} \ll \PP(\cdot |\X, S^{c}). 
\end{align}

Now  \eqref{f-dis-int} implies 
\begin{multline}\label{conv-cond}
\PP * \delta_{\Unif(S)} = \int\limits_{\Conf(S^c)}  \Big(( \PP(\cdot| \mathfrak{Z}, S^c) \otimes \delta_{\mathfrak{Z}})* \delta_{\Unif(S)} \Big) [\PP]_{S^c}( d\mathfrak{Z})=
\\
= \int\limits_{\Conf(S^c)}  \Big(( \PP(\cdot| \mathfrak{Z}, S^c) * \delta_{\Unif(S)} )   \otimes \delta_{\mathfrak{Z}}   \Big) [\PP]_{S^c}(  d\mathfrak{Z}). 
\end{multline}
Combining \eqref{cond-ins} and \eqref{conv-cond}, we obtain the desired relation 
$$
\PP * \delta_{\Unif(S)} \ll \int\limits_{\Conf(S^c)}   \Big(\PP(\cdot| \mathfrak{Z}, S^c)    \otimes \delta_{\mathfrak{Z}}   \Big) [\PP]_{S^c}(  d\mathfrak{Z}) = \PP. 
$$
\end{proof}

\begin{lemma}\label{lem-abs}
Let $S \subset \R^d$ be a relatively compact open connected subset, equipped with a Radon measure $\mu$, absolutely continuous with respect to the Lebesgue measure on $S$ and whose support is the whole space $S$. Assume that $K$ is a { trace class} infinite rank positive strict contraction  on $L^2(S, \mu)$ such that
\[
K (x, y) = \sum_{k=1}^\infty \lambda_k  \phi_k(x) \overline{\phi_k(y)}, 
\]
where $0 < \lambda_k < 1$ and $\phi_k(x)$ is { real-analytic} on $x$ for all $k$. 
Then 
\[
\PP_K(\#_S = n) > 0 \an \PP_K(\cdot | \#_S =n)   \simeq \pi^{(n)}_{*} ((dV|_S)^{\otimes n}), \quad \text{for all $n\in\N$.} 
\]
\end{lemma}

\begin{proof}[Proof of Lemma \ref{lem-abs}]
To prove Lemma \ref{lem-abs} we will need the following theorem. 
\begin{theorem}[{Hough-Krishnapur-Peres-Vir{\'a}g \cite[Theorem 7]{HKPV}}]\label{thm-HKPV}
Assume that $K$ is a trace class positive strict contraction  on $L^2(S, \mu)$ such that
\[
K (x, y) = \sum_{k=1}^\infty \lambda_k  \phi_k(x) \overline{\phi_k(y)},  \quad 0 < \lambda_k < 1. 
\]
If  $I = (I_k)_{k=1}^\infty$ is a sequence of independent Bernoulli random variables satisfying
\[
\bold{P}(I_k = 1) = \lambda_k \an \bold{P}(I_k = 0) = 1 - \lambda_k, 
\] 
and  $\mathscr{X}_I$  the random configuration sampled with the determinantal measure induced by the kernel 
\[
K_I (x, y)  = \sum_{k=1}^\infty I_k \phi_k(x) \overline{\phi_k(y)},
\]
then
$
\mathcal{L}  ( \mathscr{X}_I) = \PP_{K}. 
$
\end{theorem}

We note that in Theorem \ref{thm-HKPV},  all the random objects $I = (I_k)_{k=1}^\infty$ and $\mathscr{X}_I$ are defined on a common probability space $(\Omega, \mathcal{B}, \bold{P})$. 
We construct a random configuration $\mathscr{X}_I$ with $\mathcal{L}  ( \mathscr{X}_I) = \PP_{K}$ in the way described in  Theorem \ref{thm-HKPV}:  first choose the $I_k$'s and then independently sample a discrete set with distribution $\PP_{K_I}$.  Now fix $n\in\N$ and set $C_n = \{\alpha \in  \{0, 1\}^\N:  \sum_{i = 1}^\infty \alpha_i = n\}$.  Note that $C_n$ is a {\it countable} set. Since $K$ is trace class, we have $\sum_k \lambda_k < \infty$. Therefore,  using the assumption $0 < \lambda_k < 1$,   that for any $\alpha \in C_n$, we obtain
\[
\bold{P}( I = \alpha)  = \prod_{i: \alpha_i =1} \lambda_i  \cdot \prod_{j: \alpha_j = 0} (1- \lambda_j)> 0. 
\]
In particular,  $\bold{P}( I \in C_n) > 0$ and for any $\alpha \in C_n$, we have $ \bold{P}(I = \alpha | I \in C_n ) > 0$. 

Using the equality $\mathcal{L}  ( \mathscr{X}_I) = \PP_{K}$, we obtain 
\[
\PP_K(\#_S = n) = \bold{P}( I \in C_n) > 0.
\]
We also have 
\[
\PP_K(\cdot | \#_S =n)  = \mathcal{L}(\mathscr{X}_I | \sum_{k=1}^\infty I_k = n) =  \mathcal{L}(\mathscr{X}_I | I \in C_n). 
\]
By the construction of $\mathscr{X}_I$, we have 
\begin{multline*}
 \mathcal{L}(\mathscr{X}_I | I \in C_n)  = \sum_{\alpha \in C_n}     \bold{P}(I = \alpha | I \in C_n )  \mathcal{L}(\mathscr{X}_I | I  = \alpha)=\\
  = \sum_{\alpha \in C_n}     \bold{P}(I = \alpha| I \in C_n) \mathcal{L}(\mathscr{X}_\alpha)
 = \sum_{\alpha \in C_n}     \bold{P}(I = \alpha | I \in C_n ) \PP_{K_{\alpha}}. 
\end{multline*}
Note that  for each $\alpha \in C_n$, the operator $K_\alpha$ is an orthogonal projection of rank $n$,  we have 
\begin{align}\label{n-rank}
\PP_{K_\alpha} =  \frac{\det(K_\alpha(x_i, x_j))_{1 \le i , j \le n} }{n!}  \cdot \prod_{i = 1}^n \frac{d\mu}{dV}(x_i) \cdot \pi^{(n)}_{*} ((dV|_S)^{\otimes n}). 
\end{align}
Since all the functions $\phi_k$'s are real analytic, the function 
\begin{align}\label{anal-density}
(x_1,\cdots, x_n) \mapsto \det(K_\alpha(x_i, x_j))_{1 \le i , j \le n}
\end{align}
is also real-analytic on the open connected set $S^n \subset (\R^d)^n$. The real-analytic function function \eqref{anal-density} is not identically zero,  hence its support must be the whole space $S^n$. Moreover, the assumption on $\mu$ implies that the support of the function $d\mu/dV$ is the whole space $S$. Therefore, for each $\alpha \in C_n$, we have 
\[
\PP_{K_\alpha} \simeq \pi^{(n)}_{*} ((dV|_S)^{\otimes n}). 
\]
This implies the desired relation 
\[
\PP_K(\cdot | \#_S =n)   \simeq \pi^{(n)}_{*} ((dV|_S)^{\otimes n}). 
\]
The proof of Lemma \ref{lem-abs} is complete. 
\end{proof}
\begin{remark}
 The relation \eqref{n-rank} still holds without the assumption  that all eigenfunctions $\phi_k$'s are  real analytic, and we still have 
\begin{align}\label{diff-non-anal}
\PP_K(\cdot | \#_S =n)   \ll \pi^{(n)}_{*} ((dV|_S)^{\otimes n}), \quad \text{for any $n\in\N$.} 
\end{align}
\end{remark}

\begin{proof}[Proof of Lemma \ref{lem-ins-palm}]
First of all, for any non-empty  relatively compact and connected open subset $B\subset U$, by Lemma \ref{lem-num-ins},  we have  $
K_\omega^{[\X, B^c]}\ne 0$  for $\PP_{K_\omega}$-almost every $\X$. By Corollary \ref{cor-rank},  $\rank(K_\omega^{[\X, B^c]}) = \infty$ for $\PP_{K_\omega}$-almost every $\X$. 
Note that  $K_\omega^{[\X, B^c]}$ is a trace class positive strict contraction on $L^{2}(B,\omega dV)$ and, by Lemma \ref{lem-real-anal}, the kernel function $(z, w) \mapsto K_\omega^{[\X, B^c]}(z, w)$ is real analytic. Lemma \ref{lem-abs} implies that for $\PP_{K_\omega}$-almost every $\X$, 
\begin{align}\label{diff-cond}
 \PP_{K_\omega^{[\X, B^c]}} (\cdot| \#_B= n)  \simeq \pi^{(n)}_{*} ((dV|_B)^{\otimes n}), \quad \text{for any $n\in\N$.}
\end{align}
Since $ \PP_{K_\omega}(\cdot| \X, B^c) =  \PP_{K_\omega^{[\X, B^c]}}$ for $\PP_{K_\omega}$-almost every $\X$, Lemma \ref{lem-nice} implies that $\PP_{K_\omega}$ is  insertion tolerant. 
 The proof of insertion tolerance tolerance for  $\PP_{K_\omega}^{\mathfrak{p}}$ is the same.
\end{proof}

\subsection{Diffusive property}\label{sec-diff}
For proving Lemma \ref{lem-re-f}, we need the following Lemma \ref{lem-S}.  By the natural identification  $\C^d \subset \R^{2d}$, the domain $U\subset \C^d$ is identified with a subset of $\R^{2d}$.   Recall that by  dyadic cubes we mean  { open} subsets of $\R^{2d}$ of the form
$
v+ \Delta_n,
$
where $v \in (2^{-n} \Z)^{2d}$,  $n \in \Z$ and 
\[
\Delta_n = \prod_{k=1}^{2d} (0, 2^{-n} ).
\] 
Let $\mathcal{C}$ be the collection of all unions of finitely many dyadic cubes whose closures are contained in $U$. Clearly $\mathcal{C}$ is  countable.

\begin{lemma}\label{lem-S}
Given any two coupled random configurations $\mathscr{X}$ and $\mathscr{Y}$ on $U$ defined on the same probability space  $(\Omega, \mathcal{B}, \bold{P})$,   such that 
\begin{align*}
\mathcal{L}(\mathscr{X}) = \PP_{K_\omega}, \, \mathcal{L}(\mathscr{Y}) = \PP_{K_\omega}^{\mathfrak{p}}\an \mathscr{Y}\subset \mathscr{X} \quad \text{$\bold{P}$-almost surely.}
\end{align*}
 Then there exists a $\mathcal{C}$-valued random variable $\bold{S}$, defined on the same probability space $(\Omega, \mathcal{B}, \bold{P})$ on which the random configurations $\mathscr{X}$ and $\mathscr{Y}$ are defined, such that the equality $
\mathscr{X}|_\bold{S} = \mathscr{X}\setminus \mathscr{Y}$ holds $\bold{P}$-almost surely.
\end{lemma}

\begin{proof}
Denote  $\mathscr{Z}  = \mathscr{X}\setminus \mathscr{Y}$. Recall that  $\mathscr{X}$ is locally finite.  The random configuration  $\mathscr{Z}$ is almost surely contained by $\mathscr{X}$ and by Lemma \ref{lem-Claim-A}, $\#(\mathscr{Z})<\infty$, $\bold{P}$-almost surely. Therefore,  the following random variable
\[
\mathrm{dist} (\mathscr{Z}, \mathscr{Y}) : = \inf\{| z - y| : z \in \mathscr{Z}, y \in \mathscr{Y}\}
\] 
satisfies 
\[
\mathrm{dist} (\mathscr{Z}, \mathscr{Y}) > 0 \quad \text{$\bold{P}$-almost surely.}
\]
We may define an integer-valued random variable $\bold{n}: \Omega \rightarrow \Z$ which is the unique integer such that 
\[
\frac{1}{2^{\bold{n}}}  \le  \mathrm{dist} (\mathscr{Z}, \mathscr{Y})  <   \frac{1}{2^{\bold{n-1}}}. 
\]
We can thus define measurably a finite subset $L \subset \Z^{2d}$ by
\[
L  : = \{ v \in   (2^{-n-1}\Z)^{2d}| (v + \Delta_{\bold{n}+1}) \cap \mathscr{Z} \ne \emptyset\}. 
\]
Since the intensity measure of $\mathscr{Z}$ is  
\[
(K_\omega(z, z) - K_\omega^{\mathfrak{p}}(z, z)) \omega(z) dV \ll dV,
\]
we have 
\[
\mathscr{Z} \cap  \bigcup_{n \in \Z} \bigcup_{v\in  (2^{-n-1}\Z)^{2d}}  ( v + \partial \Delta_{n +1}) = \emptyset \quad \text{$\bold{P}$-almost surely.}
\]
Introducing a $\mathcal{C}$-valued random variable $\bold{S}$ by the formula
\[
\bold{S}: = \bigcup_{v \in L }  (v +\Delta_{\bold{n}+1}),  
\]
we obtain that the equality
$
\mathscr{X}|_\bold{S} = \mathscr{X}\setminus \mathscr{Y}$ holds  $\bold{P}$-almost surely.
\end{proof}

\begin{proof}[Proof of Lemma \ref{lem-re-f}]
Denote  $\mathscr{Z}  = \mathscr{X}\setminus \mathscr{Y}$.  By Lemma \ref{lem-S}, there exists a $\mathcal{C}$-valued random variable $\bold{S}$, defined also on the same probability space $(\Omega, \mathcal{B}, \bold{P})$ on which the random configurations $\mathscr{X}$ and $\mathscr{Y}$ are defined, such that 
\[
\mathscr{X}|_\bold{S} = \mathscr{Z} \quad  \text{$\bold{P}$-almost surely.}
\] 
It follows that 
\[
\mathscr{X}|_{U\setminus \bold{S}} = \mathscr{Y} \quad  \text{$\bold{P}$-almost surely.} 
\]

For $\PP_{K_\omega}^{\mathfrak{p}}$-almost every configuration $\mathfrak{Y}$, we have 
\[
\mathcal{L}(\mathscr{Z}| \mathscr{Y} = \mathfrak{Y}) = \mathcal{L}\Big( \mathscr{X}|_\bold{S}    \Big| \mathscr{X}|_{U\setminus \bold{S}} = \mathfrak{Y} \Big)
\]
and hence for any $n\in \N$, we have
\[
\mathcal{L}(\mathscr{Z}| \mathscr{Y} = \mathfrak{Y} , \# \mathscr{Z}  = n)    = \mathcal{L}\Big( \mathscr{X}|_\bold{S}    \Big| \mathscr{X}|_{U\setminus \bold{S}} = \mathfrak{Y}, \# \mathscr{X}|_\bold{S}   = n \Big). 
\]
Since the random set $\bold{S}$ takes values in a  countable collection $\mathcal{C}$,  we have 
\begin{multline}\label{dec-cond}
 \mathcal{L}\Big( \mathscr{X}|_\bold{S}    \Big| \mathscr{X}|_{U\setminus \bold{S}} = \mathfrak{Y}, \# \mathscr{X}|_\bold{S}   = n \Big) =
\\
 =  \sum_{C\in \mathcal{C}}   \bold{P}( \bold{S} = C) \cdot \mathcal{L}\Big( \mathscr{X}|_\bold{S}    \Big| \mathscr{X}|_{U\setminus \bold{S}} = \mathfrak{Y} ,  \# \mathscr{X}|_\bold{S}   = n, \bold{S} = C\Big)=
\\
=   \sum_{C\in \mathcal{C}}   \bold{P}( \bold{S} = C) \cdot \mathcal{L}\Big( \mathscr{X}|_C    \Big| \mathscr{X}|_{U\setminus C} = \mathfrak{Y}, \# \mathscr{X}|_C  =  n, \bold{S} = C\Big). 
\end{multline} 

To prove that each summand in the right hand side of  \eqref{dec-cond} is absolutely continuous with respect to $\pi^{(n)}_{*} ((dV|_U)^{\otimes n})$, let us first note that for any $C$ with $\bold{P}(\bold{S} = C) > 0$ and $\PP_{K_\omega}^{\mathfrak{p}}$-almost every $\mathfrak{Y}$, we have
\begin{align}\label{summand-abs}
\mathcal{L}\Big( \mathscr{X}|_C    \Big| \mathscr{X}|_{U\setminus C} = \mathfrak{Y}, \# \mathscr{X}|_C  =  n, \bold{S} = C\Big) \ll  \mathcal{L}\Big( \mathscr{X}|_C    \Big| \mathscr{X}|_{U\setminus C} = \mathfrak{Y}, \# \mathscr{X}|_C  =  n \Big). 
\end{align}
Indeed, we have the identity 
\begin{multline*}
 \mathcal{L}\Big( \mathscr{X}|_C   \Big | \mathscr{X}|_{U\setminus C} = \mathfrak{Y}, \# \mathscr{X}|_C = n \Big)   = 
\\
 = \sum_{C'\in \mathcal{C}}  \bold{P}( \bold{S} = C') \cdot  \mathcal{L}\Big( \mathscr{X}|_C   \Big | \mathscr{X}|_{U\setminus C} = \mathfrak{Y}, \# \mathscr{X}|_C = n,   \bold{S} = C' \Big).
\end{multline*}
Each summand in the right hand side is absolutely continuous with respect to the left hand side , and we obtain \eqref{summand-abs} considering the summand with $C = C'$. 

It remains to prove, for  any $C$ satisfying $\bold{P}(\bold{S} = C) > 0$ and $\PP_{K_\omega}^{\mathfrak{p}}$-almost every configuration $\mathfrak{Y}$, the relation
\begin{align}\label{remain-to-prove}
\mathcal{L}\Big( \mathscr{X}|_C    \Big| \mathscr{X}|_{U\setminus C} = \mathfrak{Y}, \# \mathscr{X}|_C  =  n \Big) \ll \pi^{(n)}_{*} ((dV|_U)^{\otimes n}). 
\end{align}
 Now for any fixed $C\in \mathcal{C}$ such that  $\bold{P}(\bold{S} = C) >0$,  by the assumption $\mathcal{L}(\mathscr{X}) = \PP_{K_\omega}$, we have
\begin{align}\label{in-out-cond}
\mathcal{L}\Big( \mathscr{X}|_C   \Big | \mathscr{X}|_{U\setminus C} = \mathfrak{Y}  \Big) & = \PP_{K_\omega} (\cdot| \mathfrak{Y}, U\setminus C) =  \PP_{K_\omega^{[ \mathfrak{Y}, U\setminus C]}},
\end{align}
where we recall that the last equality above uses the description of the conditional measures of $\PP_{K_\omega}$, see \S \ref{sec-cond-rep} and the equality \eqref{cond-m-f}. The equality \eqref{in-out-cond} now implies
\begin{align}\label{in-out-number}
\mathcal{L}\Big( \mathscr{X}|_C   \Big | \mathscr{X}|_{U\setminus C} = \mathfrak{Y}, \# \mathscr{X}|_C = n \Big)  = \PP_{K_\omega^{[ \mathfrak{Y}, U\setminus C]}} (\cdot| \#_C = n). 
\end{align}
By Lemma \ref{lem-abs} and by the same argument as for obtaining \eqref{diff-cond}, we obtain 
\begin{align}\label{diff-cond-bis}
\PP_{K_\omega^{[ \mathfrak{Y}, U\setminus C]}} (\cdot| \#_C = n) \simeq \pi^{(n)}_{*} ((dV|_C)^{\otimes n}) \ll \pi^{(n)}_{*} ((dV|_U)^{\otimes n}). 
\end{align}
Combining \eqref{in-out-number} and \eqref{diff-cond-bis}, we get
\[
\mathcal{L}\Big( \mathscr{X}|_C   \Big | \mathscr{X}|_{U\setminus C} = \mathfrak{Y}, \# \mathscr{X}|_C = n \Big)   \ll \pi^{(n)}_{*} ((dV|_U)^{\otimes n}). 
\]
Since $n$ is arbitrary, we obtain the desired relation 
\begin{multline*}
\mathcal{L}(\mathscr{Z}| \mathscr{Y}  = \mathfrak{Y})  =   \mathcal{L}\Big( \mathscr{X}|_C   \Big | \mathscr{X}|_{U\setminus C} = \mathfrak{Y}) =
\\
 = \sum_{n =0}^\infty  \bold{P}( \mathscr{X}|_C = n) \cdot \mathcal{L}\Big( \mathscr{X}|_C   \Big | \mathscr{X}|_{U\setminus C} = \mathfrak{Y}, \# \mathscr{X}|_C = n \Big)  \ll
\\
 \ll   \delta_{\emptyset}  + \sum_{n=1}^\infty \pi^{(n)}_{*} ((dV|_U)^{\otimes n}) = \sigma_U. 
\end{multline*}
\end{proof}

\begin{remark}
We emphasize that in the proof of  Lemma \ref{lem-re-f},  almost sure equalities and the resulting exclusion of null sets 
 is only used  countably many times. 
\end{remark}

\subsection{Proof of Theorem \ref{thm-q-hol}}

The proof of Theorem \ref{thm-q-hol} is similar to that of Theorem \ref{thm-palm-eq}. 
Recall that  $H \subset A_q^2(D,\omega)$ is assumed to be  a non-zero closed subspace and an $H^\infty(D)$-sub-module of $A_q^2(D,\omega)$. Under the hypothesis of Theorem \ref{thm-q-hol}, the algebra $H^\infty(D)$ has infinite dimension. An application of Lemma  \ref{lem-strict-c} yields the deletion tolerance of the determinantal  point process $\PP_{\Pi_H}$, and the same argument as in \S \ref{sec-del} yields that all the reduced Palm measures of $\PP_{\Pi_H}$ of arbitrary orders are absolutely continuous with respect to $\PP_{\Pi_H}$. 

For the converse relation $\PP_{\Pi_H} \ll \PP_{\Pi_H}^{\mathfrak{p}}$, using similar arguments as in \S \ref{sec-dif-direction},  we only need to show that $ \PP_{\Pi_H}^{\mathfrak{p}}$ is insertion tolerant and for $ \PP_{\Pi_H}^{\mathfrak{p}}$-almost every $\X$,  the conditional kernel  $\Big( \Pi_H^{\mathfrak{p}}\Big)^{[\X, B^c]}$
is real analytic for any relatively compact connected open subset $B \subset D$. Recall that a function $f: D\rightarrow \C$ is $q$-holomorphic if and only if 
\[
f(z) = \sum_{j=0}^{q-1} \bar{z}^j f_j(z),
\]
with $f_1, \cdots, f_q$ all holomorphic and the fact that uniform convergence on compact subsets preserves the class of $q$-holomorphic functions, see Balk \cite[p. 206]{Balk}.  The argument in the proof of Lemma \ref{lem-real-anal}, applied to our context, yields that
\[
\Big( \Pi_H^{\mathfrak{p}}\Big)^{[\X, B^c]}(z, \bar{w}) = \sum_{i, j =0}^{q-1} \bar{z}^i \bar{w}^j \sum_{k=0}^\infty \phi_{k, j} (z)  \overline{\phi_{k, j} (\bar{w})}, 
\]
with $\phi_{k, j}$  holomorphic functions on $B$ and the convergence taking place
 uniformly on any compact subsets of $
B\times B$. Therefore, the function  $(z, w) \mapsto ( \Pi_H^{\mathfrak{p}})^{[\X, B^c]}(z, w)$ is indeed real-analytic, and the   relation $\PP_{\Pi_H} \ll \PP_{\Pi_H}^{\mathfrak{p}}$ follows.

{\bf{Acknowledgements.}} 
We are deeply grateful to Alexei Klimenko for useful discussions and very helpful comments. The research of A. Bufetov and S. Fan on this project has received funding from the European Research Council (ERC) under the European Union's Horizon 2020 research and innovation programme under grant agreement No 647133 (ICHAOS). A. Bufetov has also been funded by the Grant MD 5991.2016.1 of the President of the Russian Federation, by  the Russian Academic Excellence Project `5-100' and by the Chaire Gabriel Lam\'e at the Chebyshev Laboratory of the SPbSU, a joint initiative of the French Embassy in the Russian Federation and the Saint-Petersburg State
University.  Y. Qiu is supported by the grant IDEX UNITI-ANR-11-IDEX-0002-02, financed by Programme ``Investissements d'Avenir'' of the Government of the French Republic managed by the French National Research Agency. Part of this work was carried out at the Institut Henri Poincar{\'e} and at the Centre international de rencontres math{\'e}matiques in the framework of the  CIRM ``recherche en petits groupes'' programme. We are deeply grateful to these institutions for their warm hospitality.

%
%

\begin{thebibliography}{10}

\bibitem{Balk}
M.~B. Balk.
\newblock Polyanalytic functions and their generalizations. 
\newblock Complex analysis, I, 195--253, 
Encyclopaedia Math. Sci., 85, Springer, Berlin, 1997. 




\bibitem{Bergman-book}
S. Bergman.
\newblock {\em The kernel function and conformal mapping}.
\newblock American Mathematical Society, Providence, R.I., revised edition,
  1970.
\newblock Mathematical Surveys, No. V.

\bibitem{Buf-inf-1}
A.~I. Bufetov.
\newblock Infinite determinantal measures and the ergodic decomposition of
  infinite {P}ickrell measures. {I}. {C}onstruction of infinite determinantal
  measures.
\newblock {\em Izv. Ross. Akad. Nauk Ser. Mat.}, 79(6):18--64, 2015.

\bibitem{BQS}
A.~I. Bufetov.
\newblock {Q}uasi-{S}ymmetries of {D}eterminantal {P}oint {P}rocesses.
\newblock {\em arXiv:1409.2068}.


\bibitem{QB3}
A.~I. Bufetov and Y. Qiu.
\newblock   {D}eterminantal point processes associated with {H}ilbert spaces of holomorphic functions.
\newblock {\em  Commun. Math. Phys.}, 351(2017), no.1, 1-44.




\bibitem{BQS16}
A.~I. Bufetov, Y. Qiu and A. Shamov.
\newblock Kernels of conditional determinantal measures.
\newblock {\em arXiv:1612.06751}.

\bibitem{DV-1}
D.~J. Daley and D.~Vere-Jones.
\newblock {\em An introduction to the theory of point processes. {V}ol. {I}}.
\newblock Probability and its Applications (New York). Springer-Verlag, New
  York, second edition, 2003.
\newblock Elementary theory and methods.

\bibitem{Ghosh-sine}
S. Ghosh.
\newblock Determinantal processes and completeness of random exponentials: the
  critical case.
\newblock {\em Probability Theory and Related Fields}, pages 1--23, 2014.

\bibitem{Ghosh-rigid}
S. Ghosh and Y. Peres.
\newblock Rigidity and tolerance in point processes: {G}aussian zeros and
  {G}inibre eigenvalues,.
\newblock {\em arXiv:1211.3506, to appear in Duke Math. J.}


\bibitem{Haimi-Hedenmalm-JSP}
A. Haimi and H. Hedenmalm.
\newblock The polyanalytic Ginibre ensembles. 
\newblock {\em J. Stat. Phys.}, 153 (2013), no. 1, 10--47.




\bibitem{Haimi-Hedenmalm}
A. Haimi and H. Hedenmalm.
\newblock Asymptotic expansion of polyanalytic Bergman kernels. 
\newblock {\em J. Funct. Anal.}, 267 (2014), no. 12, 4667--4731. 


\bibitem{HolSoo}
A.~E. Holroyd and T. Soo.
\newblock Insertion and deletion tolerance of point processes.
\newblock {\em Electron. J. Probab.}, 18:no. 74, 24, 2013.

\bibitem{HKPV}
J.~Ben Hough, M. Krishnapur, Y. Peres and B. Vir\'ag.
\newblock Determinantal processes and independence.
\newblock {\em Probab. Surv.}, 3:206--229, 2006.



\bibitem{Kallenberg}
O. Kallenberg.
\newblock {\em Random measures}.
\newblock Akademie-Verlag, Berlin; 4th ed., 1986.

\bibitem {khin}  A.Ya. Khintchine. 
\newblock {\em Mathematical methods of queuing theory,} 
\newblock Proceedings of the Steklov Institute, 1955, vol.  49, pp. 3--122. 




\bibitem{Krantz}
S.~G. Krantz.
\newblock {\em Geometric analysis of the {B}ergman kernel and metric}, volume
  268 of {\em Graduate Texts in Mathematics}.
\newblock Springer, New York, 2013.


\bibitem{Lindvall}
T. Lindvall.
\newblock On Strassen's theorem on stochastic domination.
\newblock {\em Electr. Comm. Prob.}, 4 (1999), 51--59. 

\bibitem{Lyons-ICM}
R. Lyons.
\newblock Determinantal probability: basic properties and conjectures.
\newblock In {\em Proc. International Congress of Mathematicians 2014},
  volume~IV, pages 137--161. Seoul, Korea, 2014.

\bibitem{Macchi-DP}
O. Macchi.
\newblock The coincidence approach to stochastic point processes.
\newblock {\em Advances in Appl. Probability}, 7:83--122, 1975.



\bibitem{OS}
H. Osada and T. Shirai, 
\newblock Absolute continuity and singularity of {P}alm measures of the
  {G}inibre point process.
 \newblock {\em Probab. Theory Related Fields},  165 (2016), no. 3-4, 725--770. 
  
\bibitem  {palm} C. Palm. 
\newblock {\em Intensit{\"a}tsschwankungen im Fernsprechverkehr}.
\newblock  Ericsson Technics, 1943, 44, 1-189.




\bibitem{PV-acta}
Y. Peres and B. Vir{\'a}g.
\newblock Zeros of the i.i.d.\ {G}aussian power series: a conformally invariant
  determinantal process.
\newblock {\em Acta Math.}, 194(1):1--35, 2005.

\bibitem{Qiu-Adv}
Y. Qiu.
\newblock Infinite random matrices and ergodic decomposition of finite and
  infinite {H}ua--{P}ickrell measures.
\newblock {\em Adv. Math.}, 308:1209--1268, 2017.

\bibitem{Roh-meas}
V.~A. Rohlin.
\newblock On the fundamental ideas of measure theory.
\newblock {\em Amer. Math. Soc. Translation}, 1952(71):55, 1952.

\bibitem{ST-DPP}
T. Shirai and Y. Takahashi.
\newblock Fermion process and {F}redholm determinant.
\newblock In {\em Proceedings of the {S}econd {ISAAC} {C}ongress, {V}ol. 1
  ({F}ukuoka, 1999)}, volume~7 of {\em Int. Soc. Anal. Appl. Comput.}, pages
  15--23. Kluwer Acad. Publ., Dordrecht, 2000.

\bibitem{ST-palm}
T. Shirai and Y. Takahashi.
\newblock Random point fields associated with certain {F}redholm determinants.
  {I}. {F}ermion, {P}oisson and boson point processes.
\newblock {\em J. Funct. Anal.}, 205(2):414--463, 2003.

\bibitem{Soshnikov-DP}
A. Soshnikov.
\newblock Determinantal random point fields.
\newblock {\em Uspekhi Mat. Nauk}, 55(5(335)):107--160, 2000.

\bibitem{Strassen-coupling}
V.~Strassen.
\newblock The existence of probability measures with given marginals.
\newblock {\em Ann. Math. Statist.}, 36:423--439, 1965.





\end{thebibliography}

\def\cprime{$'$} \def\cydot{\leavevmode\raise.4ex\hbox{.}}

\end{document}